\documentclass[final,twoside,11pt]{entics} 
\usepackage{enticsmacro}

\usepackage{amsmath, amssymb}
\usepackage{marginnote}

\usepackage[utf8]{inputenc}

\usepackage{proof}
\usepackage{adjustbox}

\usepackage{tikz}
\usepackage{tikz-cd}
\usetikzlibrary{matrix}
\usetikzlibrary{arrows}
\usetikzlibrary{arrows.meta}
\usetikzlibrary{decorations.pathmorphing,shapes}
\usetikzlibrary{decorations.markings}

\tikzset{spanmap/.style={
            decoration={markings,
            mark= at position 0.5 with {
                  \node[transform shape] (tempnode) {$|$};
                  }
              },
              postaction={decorate}
}
}

\tikzset{doublespanmap/.style={
            decoration={markings,
            mark= at position 0.5 with {
                  \node[transform shape] (tempnode) {$||$};
                  }
              },
              postaction={decorate}
}
}

\newcommand\Base{{\mathcal T}}
\newcommand\Der{{\mathcal D}}
\newcommand\bigcat[1]{\mathrm{#1}}
\newcommand\Set{\bigcat{Set}}
\newcommand\FinSet{\bigcat{FinSet}}
\newcommand\Span[1]{\bigcat{Span}(#1)}
\newcommand\SubSet{\bigcat{Subset}}

\newcommand\Cat{\bigcat{Cat}}

\newcommand\Dist{\bigcat{Dist}}

\newcommand\TranSet{\mathit{Tran}}
\newcommand\TranFunc{\delta}
\newcommand\Monoid[1][]{\mathcal{B}_{#1}}

\newcommand\set[1]{\{\,{#1}\,\}}
\newcommand\refs{\mathbin{\sqsubset}}

\newcommand\Words[1]{\mathcal{W}[#1]}
\newcommand\Wordsonly{\mathcal{W}}
\newcommand\Derivs[1]{\mathcal{D}[#1]}

\newcommand\defin[1]{\textbf{#1}}

\newcommand\Contour[1]{\mathcal{C}[#1]}

\newcommand\UnivGrammar[1]{\mathsf{Univ}_{#1}}
\newcommand\Dyck[1]{\mathsf{Dyck}_{#1}}

\newcommand{\Ccategory}{\mathcal{C}}
\newcommand{\Dcategory}{\mathcal{D}}

\newcommand{\Ooperad}{\mathcal{O}}

\newcommand{\Poperad}{\mathcal{P}}
\newcommand{\Qoperad}{\mathcal{Q}}
\newcommand{\Qcategory}{\mathcal{Q}}

\newcommand{\Sspecies}{\mathcal{S}}
\newcommand{\Rspecies}{\mathcal{R}}

\newcommand\FreeOperad[1][]{\mathop{\mathsf{Free}_{#1}}}
\newcommand\ForgetOperad[1]{{#1}}
\newcommand\ForgetOperadFunctor{\mathop{\mathsf{Forget}}}

\makeatletter
\newcommand\FactorOperad[2][\@nil]{\def\tmp{#1}\ifx\tmp\@nnil\mathsf{Fact}_{\,#2}\else\mathsf{Fact}_{\,#1,#2}\fi}
\makeatother

\newcommand\interval[1]{\mathrm{I}{#1}}

\newcommand\TerminalCat{\mathbf{1}}

\newcommand\TerminalSpecies{\mathbb{N}}

\newcommand\const[1]{{#1}_0}
\newcommand\elts{\mathop{\mathrm{el}}}

\newcommand\Nonterminals{N}
\newcommand\Productions{P}
\newcommand\Sentence{S}

\newcommand\Lang[1]{L_{#1}}
\newcommand\source{\mathsf{source}}
\newcommand\target{\mathsf{target}}

\newcommand{\imagefunctor}{\mathop{\mathrm{im}}}

\newcommand{\id}[1][]{id_{#1}}

\newcommand{\Species}[1]{\bigcat{Species}_{#1}}
\newcommand{\Operads}[1]{\bigcat{Operad}_{#1}}

\newcommand{\Arrowcatof}[1]{{#1}^{\to}}

\newcommand\op{{\mathrm{op}}}

\newcommand\node[1][]{\mathbin{\mathrm{node}}_{#1}}

\newcommand{\speciesunit}[1]{\mathcal{I}_{#1}}

\newcommand{\cod}{\mathrm{cod}}
\newcommand{\dom}{\mathrm{dom}}

\newcommand{\spanmap}{{\,\longrightarrow\hspace{-1.15em}|\hspace{.95em}}}

\newcommand\bin{{\mathsf{bin}}}
\newcommand\nodes{{\mathsf{nodes}}}

\newcommand\colors{{\mathsf{colors}}}

\newcommand\tagU{u}
\newcommand\tagD{d}

\newif\ifwithAppendices
\withAppendicestrue

\def\conf{MFPS 2022} 	
\volume{1}			
\def\volu{1}			
\def\papno{11}			
\def\mydoi{{\normalshape \href{https://doi.org/10.46298/entics.10508}{10.46298/entics.10508}}}
\def\copynames{P.-A. Melli\'es, N. Zeilberger} 
\def\runtitle{Parsing as a Lifting Problem, and the Chomsky-Sch\"utzenberger Representation Theorem}
\def\lastname{Melli\`es and Zeilberger}
\begin{document}
\begin{frontmatter}
  \title{Parsing as a Lifting Problem and the
  \\
  Chomsky-Sch{\"u}tzenberger Representation Theorem}

  \author{Paul-Andr\'e Melli\`es\thanksref{paulthanks}}
  \address{IRIF, Université Paris Cité, CNRS, Inria, Paris, France} 
  \thanks[paulthanks]{Email:  \href{mailto:paul-andre.mellies@cnrs.fr}{\texttt{\normalshape paul-andre.mellies@cnrs.fr}}. Partially supported by ANR ReciProg (ANR-21-CE48-0019).}
  
  \author{Noam Zeilberger\thanksref{noamthanks}}
  \address{LIX, École Polytechnique, Palaiseau, France}
  \thanks[noamthanks]{Email: \href{mailto:noam.zeilberger@lix.polytechnique.fr} {\texttt{\normalshape noam.zeilberger@lix.polytechnique.fr}}. Partially supported by ANR LambdaComb (ANR-21-CE48-0017).}

  \begin{abstract}
    Building on our work on type refinement systems, we continue developing the thesis that many kinds of deductive systems may be usefully modelled as functors and derivability as a lifting problem, focusing in this work on derivability in context-free grammars.
We begin by explaining how derivations in any context-free grammar may be naturally encoded by a functor of operads from a freely generated operad into a certain ``operad of spliced words''.
This motivates the introduction of a more general notion of context-free grammar over any category, defined as a finite species $\Sspecies$ equipped with a color denoting the start symbol and a functor of operads
$p : \FreeOperad{\Sspecies} \to \Words{\Ccategory}$ into the \emph{operad of spliced arrows} in $\Ccategory$, generating a context-free language of arrows.
We show that many standard properties of context-free grammars can be formulated within this framework, thereby admitting simpler analysis, and that usual closure properties of context-free languages generalize to context-free languages of arrows.
One advantage of considering parsing as a lifting problem is that it enables a dual fibrational perspective on the functor $p$ via the notion of \emph{displayed operad}, corresponding to a lax functor of operads $\Words{\Ccategory} \to \Span{\Set}$.
We show that displayed free operads admit an explicit inductive definition, using this to give a reconstruction of Leermakers' generalization of the CYK parsing algorithm.
We then turn to the Chomsky-Schützenberger Representation Theorem.
We start by explaining that a non-deterministic finite state automaton over words, or more generally over arrows of a category, can be seen as a category $\Qcategory$ equipped with a pair of objects denoting initial and accepting states and a functor of categories $\Qcategory \to \Ccategory$ satisfying the unique lifting of factorizations (ULF) property and the finite fiber property, recognizing a regular language of arrows.
Then, we explain how to extend this notion of automaton to functors of operads, which generalize tree automata, allowing us to lift an automaton over a category to an automaton over its operad of spliced arrows.
We show that every context-free grammar over a category can be pulled back along a non-deterministic finite state automaton over the same category, and hence that context-free languages are closed under intersection with regular languages.
The last and important ingredient is the identification of a left adjoint $\Contour{-}:\Operads{}\to\Cat$ to the operad of spliced arrows functor $\Words{-}:\Cat\to\Operads{}$.
This construction builds the contour category~$\Contour{\Ooperad}$ of any operad~$\Ooperad$, whose arrows have a geometric interpretation as ``oriented contours'' of operations.
A direct consequence of the contour / splicing adjunction is that every pointed finite species induces a universal context-free grammar, generating a language of tree contour words.
Finally, we prove a generalization of the Chomsky-Schützenberger Representation Theorem, establishing that any context-free language of arrows over a category~$\Ccategory$ is the functorial image of the intersection of a $\Ccategory$-chromatic tree contour language and a regular language.

  \end{abstract}
\end{frontmatter}

\section{Introduction}
\label{section/introduction}

In ``Functors are Type Refinement Systems'' \cite{mz15popl}, we argued for the idea that rather than being modelled merely as categories, type systems
should be modelled as functors ${p:\Der\to\Base}$ from a category~$\Der$ whose morphisms are typing derivations
to a category~$\Base$ whose morphisms are the terms corresponding to the underlying \emph{subjects} of those derivations.
One advantage of this fibrational point of view is that the notion of typing judgment receives a simple mathematical status, as a triple~$(R,f,S)$
consisting of two objects~$R,S$ in $\Der$ and a morphism $f$ in $\Base$ such that $p(R)=\dom(f)$ and $p(S)=\cod(f)$.
The question of finding a typing derivation for a typing judgment~$(R,f,S)$ then reduces 
to the lifting problem of finding a morphism $\alpha:R\to S$ such that $p(\alpha)=f$.
We developed this perspective in a series of papers \cite{mz15popl,mz18isbell,mz16bifib}, and believe that it may be usefully applied to a large variety of deductive systems, beyond type systems in the traditional sense.
In this work, we focus on derivability in context-free grammars, a classic topic in formal language theory with wide applications in computer science.

To set the stage and motivate the overall approach, let us begin by quickly explaining how context-free grammars naturally give rise to certain functors of colored operads $\Der\to\Base$.
We will assume that the reader is already familiar with context-free grammars and languages \cite{SippuSoisoiParsingTheory1} as well as with operads or multicategories \cite[Ch.~2]{LeinsterHOHC}. 
Note that ``multicategory'' and ``colored operad'' are two different names in the literature for the same concept, and in this paper we will often just use the word operad, it being implicit that operads always carry a (potentially trivial) set of colors.
We write $f \circ (g_1,\dots,g_n)$ for parallel composition of operations in an operad, and $f \circ_i g$ for the partial composition of $g$ into $f$ after the first $i$ inputs.

Classically, a context-free grammar is defined as a tuple $G = (\Sigma,\Nonterminals,\Sentence,\Productions)$ consisting of
a finite set $\Sigma$ of terminal symbols,
a finite set $\Nonterminals$ of non-terminal symbols,
a distinguished non-terminal $\Sentence \in \Nonterminals$ called the start symbol,
and a finite set $\Productions$ of production rules of the form $R\to\sigma$ where $R\in\Nonterminals$
and $\sigma\in(\Nonterminals\cup\Sigma)^{\ast}$ is a string of terminal or non-terminal symbols.
Observe that any sequence $\sigma$ on the right-hand side of a production can be factored as $\sigma=w_0 R_1 w_1 \dots R_n w_n$ where
$w_0,\dots,w_n$ are words of terminals and $R_1,\dots,R_n$ are non-terminal symbols.
We will use this simple observation in order to capture derivations in context-free grammars by functors of operads $\Der\to\Base$ from an operad $\Der$ whose colors are non-terminals to a certain monochromatic operad $\Base=\Words{\Sigma}$ that we like to call the \defin{operad of spliced words} in $\Sigma$.
The $n$-ary operations of $\Words{\Sigma}$ consist of sequences $w_0{-}w_1{-}\dots{-}w_n$ of $n+1$ words in $\Sigma^*$ separated by $n$ \emph{gaps} notated with the $-$ symbol, with composition defined simply by ``splicing into the gaps'' and interpreting juxtaposition by concatenation in $\Sigma^*$.
For example, the parallel composition of the spliced word $a{-}b{-}c$ with the pair of spliced words $d{-}e{-}f$ and $\epsilon{-}a$ is defined as
$(a{-}b{-}c) \circ (d{-}e{-}f,\epsilon{-}a) = ad{-}e{-}fb{-}ac$.
The identity operation is given by the spliced word $\epsilon{-}\epsilon$, and it is routine to check that the operad axioms are satisfied. 

Now, to any context-free grammar $G$ we can associate a free operad $\Derivs{G}$ that we call the (colored) \defin{operad of derivations} in $G$.
Its colors are the non-terminal symbols $R \in \Nonterminals$ of the grammar, while its operations are freely generated by the production rules, with each rule $r \in \Productions$ of the form $R \to w_0R_1w_1\dots R_nw_n$ giving rise to an $n$-ary operation $r : R_1,\dots,R_n \to R$.
These basic operations freely generate the operad $\Derivs{G}$ whose general operations $R_1,\dots,R_n \to R$ can be regarded as (potentially incomplete) parse trees with root label $R$ and free leaves labelled $R_1,\dots,R_n$, and with each node labelled by a production rule of $G$.
Moreover, this free operad comes equipped with an evident forgetful functor $\Derivs{G} \to \Words{\Sigma}$ that sends every non-terminal symbol $R$ to the unique color of $\Words{\Sigma}$, and every generating operation $r : R_1,\dots,R_n \to R$ as above to the spliced word $w_0{-}\dots{-}w_n$, extending to parse trees homomorphically.
See Fig.~\ref{fig:example-cfg} for an illustration.

In the rest of the paper, we will see how this point of view may be generalized to define a notion of context-free language of arrows 
between two objects~$A$ and~$B$ in any category~$\Ccategory$, by first introducing a certain operad $\Words{\Ccategory}$ of \emph{sliced arrows in $\Ccategory$}.
We will see that many standard concepts and properties of context-free grammars and languages can be formulated within this framework, thereby admitting simpler analysis, and that parsing may indeed be profitably considered from a fibrational perspective, as a lifting problem along a functor from a freely generated operad.
We will also develop a notion of non-deterministic finite state automaton and regular language of arrows in a category, and show that context-free languages are closed
under intersection with regular languages.
Finally, we will establish a categorical generalization of the Chomsky-Sch{\"u}tzenberger representation theorem,
relying on a fundamental adjunction between categories and operads that we call the contour / splicing adjunction.

\begin{figure}
  \includegraphics[width=\textwidth]{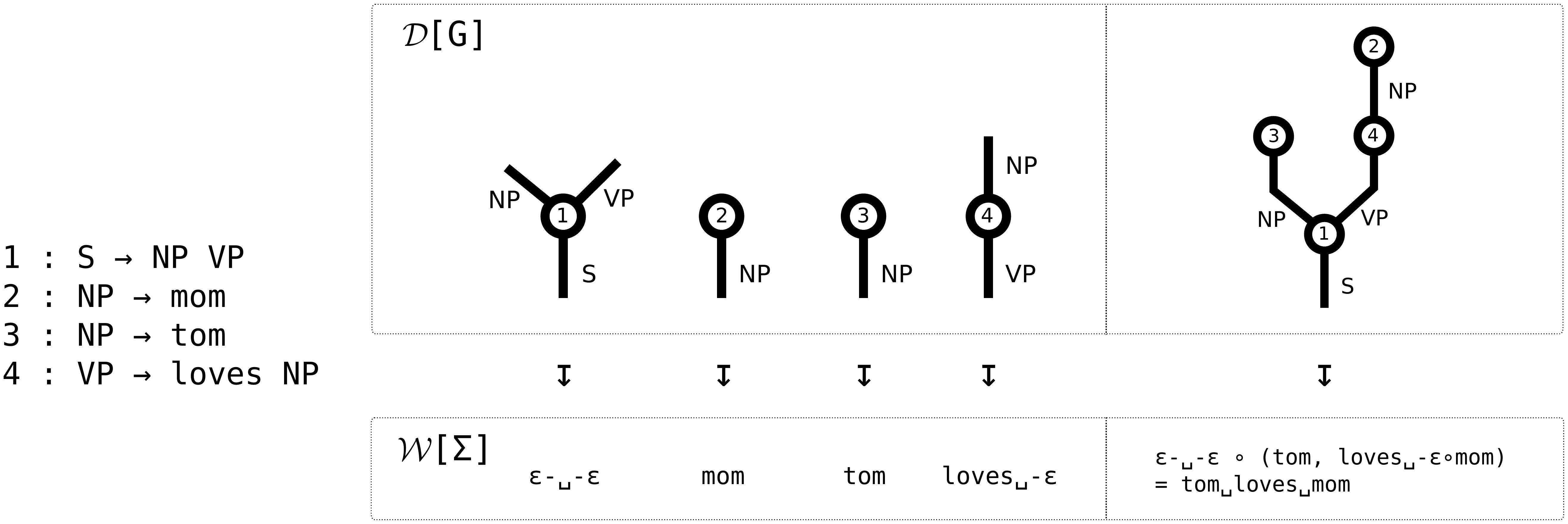}
  \caption{Example of a context-free grammar and the corresponding functor $\Derivs{G} \to \Words{\Sigma}$, indicating the action of the functor on the generating operations of $\Derivs{G}$ as well the induced action on a closed derivation.}
  \label{fig:example-cfg}
\end{figure}

\subsection*{Related work.}
The functorial perspective on context-free grammars that we just sketched and take as a starting point for this article~(\S\ref{section/context-free-languages-of-arrows}) is very similar to that of Walters in his brief ``note on context-free languages''~\cite{Walters1989}, with the main difference that we generalize it to context-free languages over any category by considering the operad of sliced arrows construction.
It is also closely related to de~Groote's treatment of CFGs in his paper introducing \emph{abstract categorial grammars}~\cite{deGroote2001acgs} and in a later article with Pogodalla~\cite{deGrootePogodalla2004}, which were developed within a $\lambda$-calculus framework rather than a categorical~/~operadic one.
The contour category construction and the contour~/~splicing adjunction between operads and categories is fundamental to our treatment of the Chomsky-Sch{\"u}tzenberger representation theorem (\S\ref{section/CSR-theorem}), and provides an unexpected geometric lens on context-free grammars, evocative of the geometry of interaction~\cite{Girard1989goi1}.
Although the adjunction is not identified, this geometric perspective is also apparent in Slavnov's recent work~\cite{Slavnov2020}, inspired both by abstract categorial grammars and by proof-nets for classical linear logic~\cite{Girard1987}, wherein he constructs a compact closed monoidal category of \emph{word cobordisms} reminiscent of the operad of spliced words.
%

The fibrational perspective on non-deterministic finite state automata as finitary ULF functors that we take in the middle of this article (\S\ref{section/ndfa}) is also similar in spirit to (and roughly dual to) Colcombet and Petri\c{s}an's proposal~\cite{ColcombetPetrisan2020} for modelling various forms of automata as functors.
Our approach is motivated by the desire to place both context-free grammars and non-deterministic finite state automata within a common framework, facilitating for example taking the intersection of a context-free language with a regular language.
Our main goal is to develop a unified framework for type systems and other deductive systems,
which would benefit from the classical body of work on context-free languages and automata theory, and in a future article we intend to consider parsing from left to right \cite{Knuth1965,Earley1970}.

\section{Context-free languages of arrows in a category}
\label{section/context-free-languages-of-arrows}
  
In this section we explain how the functorial formulation of context-free grammars discussed in the Introduction extends naturally to context-free grammars over any category, which at the same time leads to a simplification of the classical treatment of context-free languages while also providing a useful generalization.
First, we need to explain how the operad $\Words{\Sigma}$ of spliced words mentioned in the Introduction generalizes to define an operad $\Words{\Ccategory}$ of spliced arrows over any category $\Ccategory$.

\subsection{The operad of spliced arrows of a category}
\label{section/spliced-arrows}

\begin{definition}\label{definition/spliced-arrows}
Let $\Ccategory$ be a category.
The \defin{operad $\Words{\Ccategory}$ of spliced arrows in $\Ccategory$} is defined as follows:
\begin{itemize}
\item its colors are pairs $(A,B)$ of objects of $\Ccategory$;
\item its $n$-ary operations $(A_1,B_1),\dots,(A_n,B_n) \to (A,B)$ consist of sequences $w_0{-}w_1{-}\dots{-}w_n$ of $n+1$ arrows in $\Ccategory$ separated by $n$ gaps notated $-$, where each arrow must have type $w_i : B_i \to A_{i+1}$ for $0 \le i \le n$, under the convention that $B_0 = A$ and $A_{n+1} = B$;
\item composition of spliced arrows is performed by ``splicing into the gaps'': formally, the partial composition $f \circ_i g$ of a spliced arrow $g = u_0{-}\dots{-}u_m$ into another spliced arrow $f = w_0{-}\dots{-}w_n$ is defined by substituting $g$ for the $i$th occurrence of $-$ in $f$ (starting from the left using 0-indexing) and interpreting juxtaposition by sequential composition in $\Ccategory$ (see Fig.~\ref{fig:spliced-arrows} for an illustration); 
\item the identity operation on $(A,B)$ is given by $\id[A]{-}\id[B]$.
\end{itemize}
\end{definition}
\begin{figure}
  \begin{center}\includegraphics[width=0.85\textwidth]{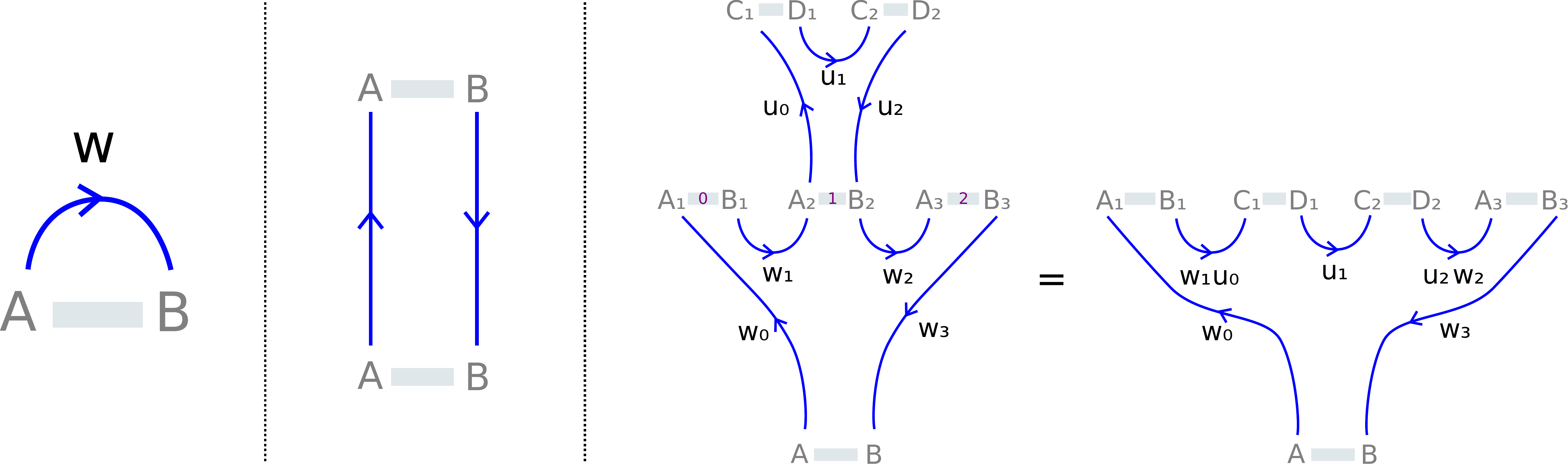}\end{center}
  \caption{Left: a constant of $\Words{\Ccategory}$. Middle: an identity operation. Right: illustration of partial composition. Here we compose an operation $g = u_0{-}u_1{-}u_2 : (C_1,D_1),(C_2,D_2) \to (A_2,B_2)$
    into $f = w_0{-}w_1{-}w_2{-}w_3 : (A_1,B_1),(A_2,B_2),(A_3,B_3) \to (A,B)$ at the gap labelled 1 to obtain the operation $f \circ_1 g = w_0{-}w_1u_0{-}u_1{-}u_2w_2{-}w_3 : (A_1,B_1),(C_1,D_1),(C_2,D_2),(A_3,B_3) \to (A,B)$.}
  \label{fig:spliced-arrows}
\end{figure}
\noindent
It is routine to check that $\Words{\Ccategory}$ satisfies the associativity and neutrality axioms of an operad, these reducing to associativity and neutrality of composition of arrows in $\Ccategory$.
Indeed, the spliced arrows operad construction defines a functor
$ \Words{-} : \Cat \to \Operads{} $
since any functor of categories $F : \Ccategory \to \Dcategory$ induces a functor of operads
$\Words{F} : \Words{\Ccategory} \to \Words{\Dcategory}$,
acting on colors by $(A,B) \mapsto (FA,FB)$ and on operations by
$w_0{-}\dots{-}w_n \mapsto Fw_0{-}\dots{-}Fw_n$.

\begin{example}\label{example/spliced-words}
  Words $w \in \Sigma^*$ may be regarded as the arrows $w : \ast \to \ast$ of a one-object category that we notate $\Monoid[\Sigma]$, with sequential composition of arrows in $\Monoid[\Sigma]$ given by concatenation.
  The operad $\Words{\Sigma}$ of spliced words in $\Sigma$ described in the Introduction is identical to the operad $\Words{\Monoid[\Sigma]}$ of spliced arrows in $\Monoid[\Sigma]$, and more generally, any monoid seen as a one-object category induces a corresponding operad of spliced words of that monoid.
\end{example}
\begin{remark}\label{remark/non-free}
  Although the one-object category $\Monoid[\Sigma]$ is a free category (being freely generated by the arrows $a : \ast \to \ast$ ranging over letters $a \in \Sigma$), this property of being freely generated does not extend to its operad of spliced words.
  Indeed, an operad of spliced arrows $\Words{\Ccategory}$ is almost \emph{never} a free operad. 
  That's because any pair of objects $A$ and $B$ induces a binary operation $\id[A]{-}\id[A]{-}\id[B] : (A,A),(A,B) \to (A,B)$,
  and any arrow $w : A \to B$ of $\Ccategory$ induces a corresponding constant $w : (A,B)$.
  Since $\id[A]{-}\id[A]{-}\id[B] \circ (\id[A],w) = w$, $\Words{\Ccategory}$ cannot be a free operad except in the trivial case where $\Ccategory$ has no objects and no arrows.
\end{remark}

\begin{example}\label{example/spliced-ordinal-sum}
  The \emph{ordinal sum} \cite{Lawvere1969ordinal} of two categories $\Ccategory$ and $\Dcategory$ may be constructed as a category $\Ccategory +_\sigma \Dcategory$ whose objects are the disjoint union of the objects of both categories, and whose arrows are the disjoint union of the arrows of both categories with an additional arrow $A \to B$ freely adjoined for every pair of objects $A \in \Ccategory$, $B \in \Dcategory$.
  The operations of the spliced arrows operad $\Words{\Ccategory +_\sigma \Dcategory}$ may be described accordingly as consisting of a (possibly empty) sequence of arrows of $\Ccategory$ followed by a (possibly empty) sequence of arrows of $\Dcategory$.
  As a special case, consider spliced arrows over the ordinal sum $\Monoid[\Sigma]^\top = \Monoid[\Sigma] +_\sigma \TerminalCat$, which is the two-object category obtained from $\Monoid[\Sigma]$ by freely adjoining an object $\top$ and an arrow $\$ : \ast \to \top$.
  The operad $\Words{\Monoid[\Sigma]^\top}$ includes operations of the form
  $ f = w_0{-}\dots{-}w_n\$ : (\ast,\ast), \dots, (\ast,\ast) \to (\ast,\top) $
  which may be seen as spliced words with an explicit ``end of input'' marker, since it is impossible to concatenate anything after the last word $w_n$ using only substitution in the operad.
  (See Example~\ref{example/cfg-with-end-of-input} below for an application of this construction.)
\end{example}

\begin{remark}
  The operad $\Words{\TerminalCat}$ of spliced arrows over the terminal category is isomorphic to the terminal operad, with a single color $(\ast,\ast)$, and a single $n$-ary operation $\id{-}\dots{-}\id : (\ast,\ast),\dots,(\ast,\ast) \to (\ast,\ast)$ of every arity $n$.
  Likewise, the operad of spliced arrows over the product of two categories decomposes as a product of spliced arrow operads $\Words{\Ccategory\times \Dcategory} \cong \Words{\Ccategory}\times\Words{\Dcategory}$.
  This might suggest that the functor $\Words{-} : \Cat \to \Operads{}$ is a right adjoint, and we will see in \S\ref{section/contour-category} that this is indeed the case.
\end{remark}


\subsection{Context-free grammars and context-free derivations over a category}
\label{section/context-free-grammar}

We already sketched in the Introduction how an ordinary context-free grammar $G = (\Sigma,\Nonterminals,\Productions,\Sentence)$ gives rise to a freely generated operad $\Derivs{G}$ equipped with a functor to the operad of spliced words $\Words{\Sigma}$, where $\Derivs{G}$ has the set of non-terminals $\Nonterminals$ as objects and operations freely generated by the productions in $\Productions$.
To make this more precise and to generalize to context-free grammars over arbitrary categories, we first need to recall the notion of a (colored non-symmetric) species, and how one gives rise to a free operad.

A \emph{colored non-symmetric species,} which we abbreviate to ``species''\footnote{Species in this sense are also called ``multigraphs'' \cite{Lambek1987} since they bear a precisely analogous relationship to multicategories as graphs do to categories, but that terminology unfortunately clashes with a different concept in graph theory.
We use ``species'' to emphasize the link with Joyal's theory of (uncolored symmetric) species \cite{Joyal1981} and also with generalized species \cite{FGHW2008}.} for short, is a tuple $\Sspecies = (C, V, i, o)$ consisting of a span of sets
$\begin{tikzcd}
  C^* & \arrow[l,"i"']V\arrow[r,"o"] & C
\end{tikzcd}$
with the following interpretation: $C$ is a set of ``colors'', $V$ is a set of ``nodes'', and the functions $i : V \to C^*$ and $o : V \to C$ return respectively the list of input colors and the unique output color of each node.
Adopting the same notation as we use for operations of an operad, we write $x : R_1,\dots,R_n \to R$ to indicate that $x \in V$ is a node with list of input colors $i(x) = (R_1,\dots,R_n)$ and output color $o(x) = R$.
However, it should be emphasized that a species by itself only contains bare coloring information about the nodes, and does not say how to compose them as operations.

We say that a species is \emph{finite} (also called \emph{polynomial} \cite{Joyal1986}) just in case both sets $C$ and $V$ are finite.

A map of species $\phi : \Sspecies \to \Rspecies$ from $\Sspecies = (C,V,i,o)$ to $\Rspecies = (D,W,i',o')$ is given by a pair $\phi = (\phi_C,\phi_V)$ of functions $\phi_C : C \to D$ and $\phi_V : V \to W$ making the diagram commute:
\[
\begin{tikzcd}
  C^*\arrow[d,"\phi_C^*"] & \arrow[l,"i"']V\arrow[r,"o"]\arrow[d,"\phi_V"] & C \arrow[d,"\phi_C"]\\
  D^* & \arrow[l,"i'"']W\arrow[r,"o'"] & D
\end{tikzcd}
\]
Equivalently, overloading $\phi$ for both $\phi_C$ and $\phi_V$, every node $x : R_1,\dots,R_n \to R$ of $\Sspecies$ must be sent to a node $\phi(x) : \phi(R_1),\dots,\phi(R_n) \to \phi(R)$ of $\Rspecies$.
Every operad $\Ooperad$ has an underlying species with the same colors and whose nodes are the operations of $\Ooperad$, and this extends to a forgetful functor $\ForgetOperadFunctor:\Operads{}\to\Species{}$ from the category of operads and functors of operads to the category of species and maps of species.
Moreover, this forgetful functor has a left adjoint 
\begin{equation}\label{equation/adjunction-species-operad}
\begin{tikzcd}
\Species{}
\arrow[rr,"\FreeOperad{}",yshift=1.1ex]
& \bot &
\Operads{}\arrow[ll,"{\ForgetOperadFunctor}",yshift=-1.1ex]
\end{tikzcd}
\end{equation}
which sends any species $\Sspecies$ to an operad $\FreeOperad{\Sspecies}$ with the same set of colors and whose operations are freely generated from the nodes of $\Sspecies$.
By the universal property of the adjoint pair, there is a natural isomorphism of hom-sets
\[
  \Operads{}(\FreeOperad{\Sspecies}, \Ooperad) \cong \Species{}(\Sspecies,\mathop{\ForgetOperadFunctor}{\Ooperad})
\]
placing functors of operads $p : \FreeOperad{\Sspecies} \to \Ooperad$ and maps of species $\phi : \Sspecies \to \mathop{\ForgetOperadFunctor}{\Ooperad}$ in one-to-one correspondence.
In the sequel, we will leave the action of the forgetful functor implicit, writing $\Ooperad$ for both an operad and its underlying species $\mathop{\ForgetOperadFunctor}\Ooperad$.

We are now ready to introduce the main definitions of this section.

\begin{definition}\label{definition/cfg-over-cat}  
  A \defin{context-free grammar of arrows} is a tuple $G = (\Ccategory,\Sspecies,S,\phi)$ consisting of a finitely generated category $\Ccategory$, a finite species $\Sspecies$ equipped with a distinguished color $S \in \Sspecies$ called the \emph{start symbol}, and a
  functor of operads $p : \FreeOperad{\Sspecies} \to \Words{\Ccategory}$.
  A color of~$\Sspecies$ is then called a \defin{non-terminal} while an operation of~$\FreeOperad\Sspecies$ is called a \defin{derivation}.
  The \defin{context-free language of arrows} $\Lang{G}$ generated by the grammar $G$ is the subset of
  arrows in $\Ccategory$ which, seen as constants of $\Words{\Ccategory}$, are in the image of constants of color $S$ in $\FreeOperad{\Sspecies}$, that is, $\Lang{G} = \set{p(\alpha) \mid \alpha : S}$.
\end{definition}
\noindent
As suggested in the Introduction and by Example~\ref{example/spliced-words}, every context-free grammar in the classical sense $G = (\Sigma,\Nonterminals,\Sentence,\Productions)$ corresponds to a context-free grammar over $\Monoid[\Sigma]$.
For instance, for the grammar in Fig.~\ref{fig:example-cfg}, the corresponding species $\Sspecies$ has three colors and four nodes, and the functor $p$ is uniquely defined by the action on the generators in $\Sspecies$ displayed in the middle of the figure.
Conversely, any finite species $\Sspecies$ equipped with a color $\Sentence\in\Sspecies$ and a functor of operads $p : \FreeOperad{\Sspecies} \to \Words{\Monoid[\Sigma]}$ uniquely determines a context-free grammar over the alphabet $\Sigma$.
Indeed, the colors of $\Sspecies$ give the non-terminals of the grammar and $\Sentence$ the distinguished start symbol, while the nodes of $\Sspecies$ together with the functor $p$ give the production rules of the grammar, with each node $x : R_1,\dots,R_n \to R$ such that $p(x) = w_0{-}w_1{-}\dots{-}w_n$ determining a context-free production rule $x : R \to w_0 R_1 w_1 \dots R_n w_n$.
\begin{proposition}
  A language $L \subseteq \Sigma^*$ is context-free in the classical sense if and only if it is the language of arrows of a context-free grammar over $\Monoid[\Sigma]$.
\end{proposition}


\noindent
An interesting feature of the general notion of context-free grammar of arrows $G = (\Ccategory,\Sspecies,\Sentence,p)$ is that the non-terminals of the grammar are \emph{sorted} in the sense that every color of $\Sspecies$ is mapped by $p$ to a unique color of $\Words{\Ccategory}$, corresponding to a pair of objects of $\Ccategory$.
Adapting the conventions from our work on type refinement systems, we sometimes write $R \refs (A,B)$ to indicate that $p(R) = (A,B)$ and say that $R$ refines the ``gap type'' $(A,B)$.
The language $\Lang{G}$ generated by a grammar with start symbol $S \refs (A,B)$ is a subset of the hom-set $\Ccategory(A,B)$.

\begin{example}\label{example/cfg-with-end-of-input}
  To illustrate some of the versatility afforded by the more general notion of context-free grammar of arrows, consider a CFG over the category $\Monoid[\Sigma]^\top$ from Example~\ref{example/spliced-ordinal-sum}.
  Such a grammar may include production rules that can only be applied upon reaching the end of the input, which is useful in practice, albeit usually modelled in an ad hoc fashion.
  For example, the grammar of arithmetic expressions defined by Knuth in the original paper on LR parsing \cite[example (27)]{Knuth1965} may be naturally described as a grammar over $\Monoid[\Sigma]^\top$, which in addition to having three ``classical'' non-terminals $E, T, P \refs (\ast,\ast)$ contains a distinguished non-terminal $S \refs (\ast,\top)$.
  Knuth's production $0 : S \to E\$$ is then just a unary node $0 : E \to S$ in $\Sspecies$, mapped by $p$ to the operation $\epsilon{-}\$ : (\ast,\ast) \to (\ast,\top)$ in $\Words{\Monoid[\Sigma]^\top}$. 
\end{example}
\noindent
More significant examples of context-free languages of arrows over categories with more than one object will be given in \S\ref{section/CSR-theorem}, including context-free grammars over the runs of finite-state automata.

Finally, let us remark that context-free grammars of arrows may be organized into a comma category, observing that a grammar $G$ may be equivalently considered as a triple of a \emph{pointed} finite species $(\Sspecies,S)$, a \emph{bipointed} finitely generated category $(\Ccategory,A,B)$, and a map of pointed operads $p : (\FreeOperad{\Sspecies},S) \to (\Words{\Ccategory},(A,B))$.
Since the operad of spliced arrows construction lifts to a functor $\Words{-} : \Cat_{\bullet,\bullet} \to \Operads{\bullet}$ sending a category $\Ccategory$ equipped with a pair of objects $A$ and $B$ to the operad $\Words{\Ccategory}$ equipped with the color $(A,B)$, and likewise the free / forgetful adjunction \eqref{equation/adjunction-species-operad} lifts to an adjunction between pointed species and pointed operads, a CFG can therefore be considered as an object of the comma category $\FreeOperad{} \downarrow \Wordsonly$.
Although we will not explore this perspective further here, let us mention that it permits another way of understanding the language of arrows generated by a grammar $G$:
as constants of $\Words{\Ccategory}$ are in bijection with arrows of $\Ccategory$,
we have a natural isomorphism $\elts \circ \Words{-} \cong \hom$ for the evident functors $\hom : \Cat_{\bullet,\bullet} \to \Set$ and $\elts : \Operads{\bullet} \to \Set$,
and $\Lang{G}$ is precisely the image of the function $\elts(p) : \elts(\FreeOperad{\Sspecies},S) \to \elts(\Words{\Ccategory},(A,B)) \cong \Ccategory(A,B)$.

\subsection{Properties of a context-free grammar and its associated language}
\label{section/properties-of-cfg-and-language}

Standard properties of context-free grammars \cite[Ch.~4]{SippuSoisoiParsingTheory1}, considered as CFGs of arrows $G = (\Monoid[\Sigma],\Sspecies,\Sentence,p)$, may be reformulated as properties of either the species $\Sspecies$, the operad $\FreeOperad{\Sspecies}$, or the functor $p : \FreeOperad{\Sspecies} \to \Words{\Monoid[\Sigma]}$, with varying degrees of naturality:

\begin{itemize}
\item $G$ is \emph{linear} just in case $\Sspecies$ only has nodes of arity $\le 1$.
  It is \emph{left-linear} (respectively, \emph{right-linear}) just in case it is linear and every unary node $x$ of $\Sspecies$ is mapped by $p$ to an operation of the form $\epsilon{-}w$ (resp.~$p(x) = w{-}\epsilon$). 
\item $G$ is in \emph{Chomsky normal form} if $\Sspecies$ only has nodes of arity 2 or 0, the color $\Sentence$ does not appear as the input of any node, every binary node is mapped by $p$ to $\epsilon{-}\epsilon{-}\epsilon$ in $\Words{\Monoid[\Sigma]}$, and every nullary node is mapped to a letter $a \in \Sigma$, unless $R = \Sentence$ in which case it is possible that $p(x) = \epsilon$.
  (This last condition can be made more natural by considering $G$ as a context-free grammar over $\Monoid[\Sigma]^\top$ with $S \refs (\ast,\top)$, see~Example~\ref{example/cfg-with-end-of-input} above.)
\item
  $G$ is \emph{bilinear} (a generalization of Chomsky normal form \cite{LangeLeiss09,Leermakers1989}) iff $\Sspecies$ only has nodes of arity $\le 2$.
\item
  $G$ is \emph{unambiguous} iff for any pair of constants $\alpha, \beta : S$ in $\FreeOperad{\Sspecies}$, if $p(\alpha) = p(\beta)$ then $\alpha = \beta$.
  Note that if $p$ is faithful then $G$ is unambiguous, although faithfulness is a stronger condition in general.
\item
  A non-terminal $R$ of $G$ is \emph{nullable} if there exists a constant $\alpha : R$ of $\FreeOperad{\Sspecies}$ such that $p(\alpha) = \epsilon$.
\item
  A non-terminal $R$ of $G$ is \emph{useful} if there exists a pair of a constant $\alpha : R$ and a unary operation $\beta : R \to S$.
  Note that if $G$ has no useless non-terminals then $G$ is unambiguous iff $p$ is faithful.
\end{itemize}
Observe that almost all of these properties 
can be immediately translated to express properties of context-free grammars of arrows over \emph{any} category $\Ccategory$. 
Basic closure properties of classical context-free languages also generalize easily to context-free languages of arrows.
\begin{proposition}\label{proposition/basic-closure}
  \begin{enumerate}
  \item If $L_1,L_2 \subseteq\Ccategory(A,B)$ are context-free languages of arrows, so is their union $L_1 \cup L_2 \subseteq\Ccategory(A,B)$.
  \item If $L_1\subseteq\Ccategory(A_1,B_1), \dots,L_n\subseteq \Ccategory(A_n,B_n)$ are context-free languages of arrows, and $w_0{-}w_1{-}\dots{-}w_n : (A_1,B_1),\dots,(A_n,B_n) \to (A,B)$ is an operation of $\Words{\Ccategory}$, then the ``spliced concatenation'' $w_0 L_1 w_1 \dots L_n w_n = \set{w_0u_1w_1\dots u_n w_n \mid u_1 \in L_1,\dots, u_n \in L_n} \subseteq \Ccategory(A,B)$ is also context-free.
  \item If $L \subseteq \Ccategory(A,B)$ is a context-free language of arrows in a category $\Ccategory$ and $F : \Ccategory \to \Dcategory$ is a functor of categories, then the functorial image $F(L) \subseteq \Dcategory(F(A),F(B))$ is also context-free.
\end{enumerate}
\end{proposition}
\begin{proof}
  The proofs of (i) and (ii) are just refinements of the standard proofs for context-free languages of words, keeping track of the underlying gap types.
  For (iii), suppose given a grammar $G=(\Ccategory,\Sspecies,S,p)$ and a functor of categories $F : \Ccategory \to \Dcategory$.
  Then the grammar $F(G)$ generating the language $F(\Lang{G})$ is defined by postcomposing $p$ with $\Words{F} : \Words{\Ccategory} \to \Words{\Dcategory}$ while keeping the species $\Sspecies$ and start symbol $S$ the same, $F(G)= (\Dcategory,\Sspecies,S,p \Words{F})$.

\end{proof}
\noindent
We will see in \S\ref{section/pulling-back} that other classical closure properties also generalize to context-free languages of arrows.
Finally, we can state a \defin{translation principle} that two grammars $G_1=(\Ccategory,\Sspecies_1,S_1,p_1)$ and~$G_2=(\Ccategory,\Sspecies_2,S_2,p_2)$
over the same category have the same language whenever there is a fully faithful functor of operads $T : \FreeOperad{\Sspecies_1}\to\FreeOperad{\Sspecies_2}$
such that $p_1 = T p_2$ and $T(S_1) = S_2$.


\subsection{A fibrational view of parsing as a lifting problem}
\label{section/benabou}

We have seen how any context-free grammar $G=(\Ccategory,\Sspecies,S,p)$ gives rise to a language 
$\Lang{G} = \set{p(\alpha) \mid \alpha : S}$, corresponding to the arrows of $\Ccategory$ which, seen as constants of $\Words{\Ccategory}$, are in the image of some constant of color $S$ of the free operad $\FreeOperad{\Sspecies}$.
However, beyond characterizing the language defined by a grammar, in practice one is often confronted with a dual problem, namely that of parsing: given a word $w$, we want to compute the set of all its parse trees, or at least determine all of the non-terminals which derive it.
In our functorial formulation of context-free derivations, this amounts to computing the inverse image of $w$ along the functor $p$, i.e., the set of constants $p^{-1}(w) = \set{\alpha \mid p(\alpha) = w}$, or alternatively the set of colors in the image of $p^{-1}(w)$ along the output-color function.

To better understand this view of parsing as a lifting problem along a functor of operads, we find it helpful to first recall the correspondence between functors of categories $p : \Der \to \Base$ and lax functors $F : \Base \to \Span\Set$, where $\Span\Set$ is the bicategory whose objects are sets, whose 1-cells $S:X\,{\spanmap} Y$ are spans $X\leftarrow S\rightarrow Y$,
and whose 2-cells are morphisms of spans.
Suppose given such a functor $p : \Der \to \Base$.
To every object $A$ of $\Base$ there is an associated ``fiber'' $F_A = p^{-1}(A)$ of objects in $\Der$ living over $A$, while to every arrow $w : A \to B$ of $\Base$ there is an associated fiber $F_w = p^{-1}(w)$ of arrows in $\Der$ living over $w$, equipped with a pair of projection functions $F_A\leftarrow F_w\rightarrow F_B$ mapping any lifting $\alpha : R \to S$ of $w : A \to B$ to its source $R \in F_A$ and target $S \in F_B$.
Moreover, given a pair of composable arrows
$u:A\to B$ and $v:B\to C$ in $\Base$, there is a morphism of spans
\begin{equation}\label{equation/lax-o}
\begin{tikzcd}[column sep = 1.2em, row sep = .8em]
{F_uF_v}
\arrow[double,-implies,rr]
&&
{F_{u v}}
\quad
:
\quad
{F_{A}} \arrow[spanmap,rr]
&&
{F_{C}}
\end{tikzcd}
\end{equation}
from the composite of the spans 
$F_{u}:F_{A}\spanmap F_{B}$
and 
$F_{v}:F_{B}\spanmap F_{C}$ associated
to $u:A\to B$ and $v:B\to C$ 
to the span $F_{u v}:F_{A}\spanmap F_{C}$ associated to the composite arrow $u v:A\to C$.
This morphism of spans is realized using composition in the category $\Der$,
namely by the function taking any pair of a lifting $\alpha : R \to S$ of $u$ and a lifting $\beta : S \to T$ of $v$
to the composite $\alpha \beta : R \to T$, which is a lifting of $u v$ by functoriality $p(\alpha\beta) = p(\alpha)p(\beta)$.
Similarly, the identity arrows in the category $\Der$ define, for every object $A$ of the category $\Base$,
a morphism of spans
\begin{equation}\label{equation/lax-I}
\begin{tikzcd}[column sep = 1.2em, row sep = .8em]
{\id[F_A]}
\arrow[double,-implies,rr]
&&
{F_{\id[A]}}
\quad
:
\quad
{F_{A}} \arrow[spanmap,rr]
&&
{F_{A}}
\end{tikzcd}
\end{equation}
from the identity span $F_A\leftarrow F_A\rightarrow F_A$
to the span associated to the identity arrow $\id[A]:A\to A$.
Associativity and neutrality of composition in $\Der$
ensure that the 2-cells \eqref{equation/lax-o} and \eqref{equation/lax-I} make the diagrams below commute:
\adjustbox{scale=0.75,center}{%
\begin{tikzcd}[column sep = .8em, row sep = 1.8em]
{F_u F_v F_w}
\arrow[rr,double,-implies]
\arrow[dd,double,-implies]
&&
{F_u F_{v w}}
\arrow[dd,double,-implies]
\\
\\
{F_{u v}F_w}
\arrow[rr,double,-implies]
&&
{F_{u v w}}
\end{tikzcd}
\quad
\begin{tikzcd}[column sep = .6em, row sep = .8em]
&& {F_{u}}
\arrow[dddd,double,-]
\arrow[lldd,double,-implies]
&&
\\
\\
{F_{\id[A]}F_{u}}
\arrow[rrdd,double,-implies]
\\
\\
&& {F_{u}}
\end{tikzcd}
\quad
\begin{tikzcd}[column sep = .6em, row sep = .8em]
{F_{u}}
\arrow[dddd,double,-]
\arrow[rrdd,double,-implies]
&&
\\
\\
&&
{F_{u} F_{\id[B]}}
\arrow[lldd,double,-implies]
\\
\\
{F_{u}}
\end{tikzcd}
}
for all triples of composable arrows $u:A\to B$, $v:B\to C$ and $w:C\to D$, and therefore that 
this collection of data defines what is called a lax functor $F:\Base\to\Span{\Set}$.
In general it is \emph{only} lax, in the sense that the 2-cells $F_u F_v \Rightarrow F_{uv}$ and $id_{F_A} \Rightarrow F_{id_A}$ are not necessarily invertible.

Conversely, starting from the data provided by a lax functor $F:\Base\to\Span{\Set}$, we can define a category noted $\smallint F$ together with a functor $\pi:\smallint F\to\Base$.
The category $\smallint F$ has objects the pairs $(A,R)$ of an object~$A$ in $\Base$ and an element $R\in F_A$,
and arrows $(w,\alpha):(A,R)\to (B,S)$ the pairs of an arrow $w:A\to B$ in $\Base$ and an element $\alpha\in F_w$ mapped to $R\in F_A$ and $S\in F_B$ by the respective legs of the span $F_A\leftarrow F_w\rightarrow F_B$.
The composition and identity of the category $\smallint F$ are then given by the morphisms of spans $F_u F_v\Rightarrow F_{u v}$ and $\id[F_A] \Rightarrow F_{\id[A]}$ witnessing the lax functoriality of $F:\Base\to\Span{\Set}$.
The functor $\pi:\smallint F\to\Base$ is given by the first projection.
This construction of a category $\smallint F$ equipped with a functor $\pi:\smallint F\to\Base$ starting from a lax functor $F : \Base \to \Span\Set$
is a mild variation of B\'enabou's construction of the same starting from a lax normal functor $F : \Base^\op \to \Dist$ \cite[\S7]{Benabou2000}, which is itself a generalization of the well-known Grothendieck construction of a fibration starting from a pseudofunctor $F : \Base^\op \to \Cat$.
One can show that given a functor of categories 
$p:\Der\to\Base$, the construction
applied to the associated lax functor $F:\Base\to\Span{\Set}$
induces a category $\smallint F$ isomorphic to $\Der$,
in such a way that~$p$ coincides with the isomorphism composed with $\pi$.
Recently, Ahrens and Lumsdaine \cite{AhrensLumsdaine2019} have introduced the useful terminology ``displayed category'' to refer to this way of presenting a category $\Der$ equipped with a functor $\Der \to \Base$ as a lax functor $\Base \to \Span\Set$, with their motivations coming from computer formalization of mathematics.

The constructions which turn a functor of categories $p:\Der\to\Base$ into a lax functor $F:\Base\to\Span{\Set}$ and back into a functor $\pi : \smallint F \to \Base$ can be adapted smoothly to functors of operads, viewing $\Span\Set$ as a 2-categorical operad whose $n$-ary operations
$S:X_1,\dots,X_n\,{\spanmap} Y$ are multi-legged-spans
\adjustbox{scale=0.8,center}{%
$
\begin{tikzcd}[row sep = 0em]
X_1 \\
\vdots &
S
\arrow[lu]\arrow[ld]\arrow[r]
&
Y \\
X_n
\end{tikzcd}
$
}
or equivalently spans $X_1\times\dots\times X_n \leftarrow S\rightarrow Y$,
and with the same notion of 2-cell.
%
We will follow Ahrens and Lumsdaine's suggestion and refer to the data of such a lax functor $F : \Base \to \Span\Set$ representing an operad $\Der \cong \smallint F$ equipped with a functor $p : \Der \to \Base$ as a \defin{displayed operad}.

\subsection{An inductive formula for displayed free operads}
\label{section/magic-formula}

It is folklore that the free operad over a species $\Sspecies = (C,V,i,o)$ may be described concretely as a certain family of trees: operations of $\FreeOperad{\Sspecies}$ are interpreted as rooted planar trees whose edges are colored by the elements of $C$ and whose nodes are labelled by the elements of $V$, subject to the constraints imposed by the functions $i : V \to C^*$ and $o : V \to C$.
The formal construction of the free operad may be viewed as a free monoid construction, adapted to a situation where the ambient monoidal product (in this case, the composition product of species) is only distributive on the left, see \cite[II.1.9]{MarklSchniderStasheff} and \cite[Appendix~B]{BJT1997}.

From the perspective of programming semantics, it is natural to consider the underlying species of $\FreeOperad{\Sspecies}$ as an inductive data type, corresponding to the initial algebra for the endofunctor $W_\Sspecies$ on $C$-colored species defined by
\[
W_\Sspecies = \Rspecies \mapsto \speciesunit{} + \Sspecies \circ \Rspecies
\]
where $+$ denotes the coproduct of $C$-colored species which is constructed by taking the disjoint union of operations, while $\circ$ and $\speciesunit{}$ denote respectively the composition product of $C$-colored species and the identity species, defined as follows.
Given two $C$-colored species $\Sspecies$ and $\Rspecies$, 
the $n$-ary nodes $R_1,\dots,R_n\to R$
of $\Sspecies\circ \Rspecies$ 
are formal composites $g\bullet (f_1,\dots,f_k)$ consisting of a node
$g:S_1,\dots,S_k\to S$ of $\Sspecies$ and of a tuple of nodes
$f_1:\Gamma_1\to S_1$, $\dots$, $f_k:\Gamma_k\to S_k$ of $\Rspecies$,
such that the concatenation of the lists of colors $\Gamma_1,\dots,\Gamma_k$ is equal to the list $R_1,\dots,R_n$.
The unit $\speciesunit{}$ is the $C$-colored species with a single unary node $\ast_R:R\to R$ for every color $R\in C$,
and no other nodes.

As the initial $W_\Sspecies$-algebra, the free operad over $\Sspecies$ is equipped with a map of species
$\speciesunit{} + \Sspecies \circ \FreeOperad{\Sspecies} \longrightarrow \FreeOperad{\Sspecies}$,
which by the Lambek lemma is invertible,
with the following interpretation: any operation of $\FreeOperad{\Sspecies}$ is either an identity operation, or the parallel composition of a node of $\Sspecies$ with a list of operations of $\FreeOperad{\Sspecies}$.
Note that this interpretation also corresponds to a canonical way of decomposing trees labelled by the species $\Sspecies$, also known as \emph{$\Sspecies$-rooted trees} \cite[\S3.2]{BeLaLe1998Species}.

It is possible to derive an analogous inductive characterization of \emph{functors} $p : \FreeOperad{\Sspecies} \to \Ooperad$ from a free operad into an arbitrary operad $\Ooperad$ considered as displayed free operads, i.e., as lax functors $F : \Ooperad \to \Span\Set$ generated by an underlying map of species $\phi : \Sspecies \to \ForgetOperad{\Ooperad}$.
Two subtleties arise.
First, that the species $\Sspecies$ and the operad $\Ooperad$ may in general have a different set of colors, related by the change-of-color function $\phi_C$.
To account for this, rather than restricting the operations $+, \circ,\speciesunit{}$ to the category of $C$-colored species,
one should consider them as global functors
\[ +,\circ : \Species{} \times_{\Set} \Species{} \to \Species{}
  \qquad
  \speciesunit{} : \Set \to \Species{}
\]
on the ``polychromatic'' category of species, which respect the underlying sets of colors in a functorial way.
Second, and more significantly, the above functor $W_\Sspecies$ transports a species $\Rspecies$ living over $\ForgetOperad\Ooperad$ to a species living over $\speciesunit{} + \ForgetOperad\Ooperad \circ \ForgetOperad\Ooperad$, so that in order to obtain again a species living over $\ForgetOperad\Ooperad$ (and thus define an endofunctor) one needs to ``push forward'' along the canonical $W_{\ForgetOperad\Ooperad}$-algebra $[e,m] : \speciesunit{} + \ForgetOperad\Ooperad \circ \ForgetOperad\Ooperad \longrightarrow \ForgetOperad\Ooperad$ that encodes the operad structure of $\Ooperad$, seen as a monoid in $(\Species{},\circ,\speciesunit{})$.
A detailed proof is beyond the scope of this paper, but we nevertheless state the following:
\begin{proposition}Let $\phi : \Sspecies \to \ForgetOperad{\Ooperad}$ be a map of species from a species $\Sspecies$ into an operad $\Ooperad$, and let $p : \FreeOperad{\Sspecies} \to \Ooperad$ be the corresponding functor from the free operad.
  Then the associated lax functor $F : \Ooperad \to \Span\Set$ computing the fibers of $p$ is given by $F_A = \phi^{-1}(A)$ on colors of $\Ooperad$, and by the least family of sets $F_f$ indexed by operations $f : A_1,\dots,A_n \to A$ of $\Ooperad$ such that
\begin{equation}\label{equation/magic-formula}
  F_f \cong \sum_{{f=\id[A]}\atop {\phi(R) = A}}\id[R] + \sum_{f=g\circ (h_1,\dots,h_k)} \phi^{-1}(g)\bullet(F_{h_1},\dots,F_{h_k})
\end{equation}
  where we write $\circ$ for composition in the operad $\Ooperad$ and $\bullet$ for formal composition of nodes in $\Sspecies$ with operations in $\FreeOperad{\Sspecies}$.
  Specializing the formula to constant operations, the left summand disappears and \eqref{equation/magic-formula} simplifies to:
\begin{equation}\label{equation/magic-formula0}
  F_c \cong \sum_{c=g\circ (c_1,\dots,c_k)} \phi^{-1}(g)\bullet(F_{c_1},\dots,F_{c_k})
\end{equation}
\end{proposition}

\subsection{Application to parsing} 
\label{section/application-to-parsing}

Instantiating \eqref{equation/magic-formula0} with the underlying functor $p : \FreeOperad{\Sspecies} \to \Words{\Ccategory}$ of a context-free grammar of arrows generated by a map of species $\phi : \Sspecies \to \Words{\Ccategory}$, we immediately obtain the following characteristic formula for the family of sets of parse trees $F_w$ of an arrow $w$ in $\Ccategory$, seen as liftings of the constant $w$ in $\Words{\Ccategory}$ to a constant in $\FreeOperad{\Sspecies}$:
\begin{equation}\label{equation/magic-formula0-words}
  F_w \cong \sum_{w=w_0 u_1 w_k\dots u_n w_k} \phi^{-1}(w_0{-}w_1{-}\dots{-}w_k)\bullet(F_{u_1},\dots,F_{u_k})
\end{equation}
Taking the image along the function returning the root label of a parse tree (i.e., the underlying color of the constant in $\FreeOperad{\Sspecies}$), we get that the family of sets of non-terminals $N_w$ deriving $w$ is the least family of sets closed under the following inference rule:
\begin{equation}\label{equation/magic-formula0-words-subset}
  \infer{R \in N_w}{w=w_0 u_1 w_1\dots u_k w_k & \deduce{\phi(x) = w_0{-}w_1{-}\dots{-}w_n}{(x : R_1,\dots,R_k \to R) \in \Sspecies} & R_1 \in N_{u_1} & \dots & R_k \in N_{u_k}}
\end{equation}
This inference rule is essentially the characteristic formula expressed by Leermakers \cite{Leermakers1989} for the defining relation of the ``C-parser'', which generalizes the well-known Cocke-Younger-Kasami (CYK) algorithm.
Presentations of the CYK algorithm are usually restricted to grammars in Chomsky normal form (cf.~\cite{LangeLeiss09}), but as observed by Leermakers, the relation $N_w$ defined by \eqref{equation/magic-formula0-words-subset} can be solved effectively for any context-free grammar $G$ and given word $w = a_1 \dots a_n$ by building up a parse matrix $N_{i,j}$ indexed by the subwords $w_{i,j} = a_{i+1}\dots a_j$ for all $1 \le i \le j \le n$, yielding a cubic complexity algorithm in the case that $G$ is bilinear (cf.~\S\ref{section/properties-of-cfg-and-language}).
Moreover, by adding non-terminals, it is always possible to transform a CFG into a bilinear CFG that generates the same language, even preserving the original derivations up to isomorphism.
\begin{proposition}\label{proposition/bilinear-normal-form}
  For any context-free grammar of arrows $G = (\Ccategory, \Sspecies, S, p)$, there is a bilinear context-free grammar of arrows $G_\bin = (\Ccategory, \Sspecies_\bin, S, p_\bin)$ together with a fully faithful functor of operads
  $B : \FreeOperad{\Sspecies} \to \FreeOperad{\Sspecies_\bin}$ such that $p = B p_\bin$.
  In particular, $\Lang{G} = \Lang{G_\bin}$ by the translation principle.
\end{proposition}


\section{Non-deterministic finite state automata as finitary ULF functors over categories and operads}
\label{section/ndfa}
\subsection{Warmup: non-deterministic word automata as finitary ULF functors over categories}\label{section/warmup}
Classically, a non-deterministic finite state automaton $M=(\Sigma,Q,\TranFunc,q_0,F)$
consists of a finite alphabet~$\Sigma$, a finite set~$Q$ of states, 
a function $\TranFunc:\TranSet\to Q\times\Sigma\times Q$ from a finite set~$\TranSet$ of transitions,
an initial state~$q_0\in Q$ and a finite set of accepting states~$F\subseteq Q$.
We will focus first on the underlying ``bare'' automaton~$M = (\Sigma,Q,\TranFunc)$ where the initial and the accepting states have been removed.
Every such bare automaton~$M$ induces a functor of categories $p:\Qcategory \to \Monoid[\Sigma]$
where $\Qcategory$ is the category with the states of the automaton as objects, and with arrows freely generated
by arrows of the form $t:q\to q'$ for any transition $t\in\TranSet$ such that $\TranFunc(t)=(q,a,q')$; 
%
and where the functor $p:\Qcategory\to \Monoid[\Sigma]$ transports every transition $t:q\to q'$ with $\TranFunc(t)=(q,a,q')$ 
to the arrow $a:\ast\to\ast$ representing the letter~$a\in\Sigma$ in the category~$\Monoid[\Sigma]$. 
Under this formulation,
every arrow $\alpha:q_0\to q_f$ of the category~$\Qcategory$ describes 
a run of the automaton~$M$ over the word $w=p(\alpha):\ast\to\ast$ which starts in state~$q_0\in Q$ and ends in state~$q_f\in Q$, 
as depicted below:
\begin{center}
\begin{tikzcd}[column sep=2em,row sep=.8em]
{q_0}\arrow[dd,|->]\arrow[rrrr,dashed,"{\alpha}"] &&&& {q_f}\arrow[dd,|->]
&&
\Qcategory\arrow[dd,"{p}"]
\\
\\
\ast\arrow[rrrr,"w"] &&&& \ast
&&
\Monoid[\Sigma]
\end{tikzcd}
\end{center}
One distinctive property of the functor $p:\Qcategory \to \Monoid[\Sigma]$ is that it has
the unique lifting of factorizations (ULF) property in the sense of Lawvere and Menni~\cite{LawvereMenni2010}.
Recall that a functor of categories $p : \Der \to \Base$ has the ULF property when:
\begin{center}
\begin{tabular}{l}
\vspace{-2em}
\\
For any arrow $\alpha$ of the category $\Der$, if $p(\alpha)= uv$ for some pair of arrows~$u$ and~$v$ of the category~$\Base$,
\vspace{-.4em}
\\
there exists a unique pair of arrows~$\beta$ and~$\gamma$ in $\Der$ such that $\alpha=\beta\gamma$ and $p(\beta)=u$ and $p(\gamma)=v$.
\\
\vspace{-2em}
\end{tabular}
\end{center}
Note that a functor $p : \Der \to \Base$ has the ULF property
precisely when the structure maps of the corresponding lax functor $F:\Base\to\Span{\Set}$ discussed in \S\ref{section/benabou} are invertible, i.e., $F$ is a pseudofunctor.
The ULF property implies an important structural property of non-deterministic finite state automata:
that every arrow $\alpha:q_0\to q_f$ lying above some arrow $p(\alpha)=w$ corresponding to a run of the automaton
can be factored uniquely as a sequence of transitions along the letters of the word~$w$.
%
Conversely, we can easily establish that
\begin{proposition}
  A ULF functor $p:\Qcategory\to\Monoid[\Sigma]$ corresponds to a bare non-deterministic finite state automaton 
  precisely when the fiber $p^{-1}(\ast)$ as well as the fiber $p^{-1}(w)$ is finite for all words $w:\ast\to\ast$.
\end{proposition}
\noindent
This leads us to the following definitions.
\begin{definition}
  We say that a functor $p : \Qcategory \to \Ccategory$ is \defin{finitary} if either of the following equivalent conditions hold:
  \begin{itemize}
  \item the fiber $p^{-1}(A)$ as well as the fiber $p^{-1}(w)$ is finite for every object $A$ and arrow $w$ in the category $\Ccategory$;
  \item the associated lax functor $F:\Ccategory\to\Span{\Set}$ factors via $\Span{\FinSet}$.
  \end{itemize}
\end{definition}

\begin{definition}
  A \defin{non-deterministic finite state automaton over a category} is given by a tuple
  $M = (\Ccategory, \Qcategory, {p : \Qcategory \to \Ccategory}, q_0, q_f)$
  consisting of two categories $\Ccategory$ and $\Qcategory$, a finitary ULF functor $p:\Qcategory\to\Ccategory$,
  and a pair $q_0,q_f$ of objects of $\Qcategory$.
  An object of~$\Qcategory$ is then called a \defin{state} and an arrow of~$\Qcategory$ is called a \defin{run} of the automaton.
  The \defin{regular language of arrows} $\Lang{M}$ recognized by the automaton
  is the set of arrows $w$ in $\Ccategory$ that can be lifted along~$p$ to an arrow $\alpha:q_0\to q_f$ in $\Qcategory$, that is $\Lang{M} = \set{p(\alpha) \mid \alpha : q_0 \to q_f}$.
\end{definition}
\noindent
Note that the regular language of arrows $\Lang{M}$ recognized by an automaton $M$ is a subset of the hom-set
$\Ccategory(A,B)$, where $A=p(q_0)$ and $B=p(q_f)$.
\begin{remark}
Any non-deterministic finite state automaton $M$ in the standard sense may be converted into an
automaton with a single accepting state (and without $\epsilon$-transitions) that accepts the same language, 
\emph{except} in the case that the language contains $\epsilon$ and is not closed under concatenation.
The usual construction defines a new automaton $M'$ with an additional state $q_f$ and the same transitions as $M$, except that
every transition $q \to q'$ to an accepting state $q' \in F$ of the old automaton is replaced by a transition $q \to q_f$ in
the new automaton.
The problem arises when the initial state $q_0$ is also accepting, in which case the language accepted by $M'$ will be closed under concatenation.

Observe that this issue goes away if we instead consider the automaton obtained by transformation of $M$ as an automaton $M' = (\Monoid[\Sigma]^\top,\Qcategory',p',q_0,q_f)$ over the two-object category $\Monoid[\Sigma]^\top =  \Monoid[\Sigma] +_\sigma \TerminalCat$ defined in Example~\ref{example/spliced-ordinal-sum}.
Indeed, we can take $\Qcategory'$ and $p'$ to be defined from $\Qcategory$ and $p$ by adjoining a single object $q_f$ lying over $\top$, together with a single arrow $q' \to q_f$ lying over $\$ : \ast \to \top$ for every accepting state $q' \in F$ of $M$.
Since arrows of type $\ast \to \top$ do not compose, the aforementioned problem does not arise.
\end{remark}
\begin{proposition}
  A language $L \subseteq \Sigma^*$ is regular in the classical sense if and only if  $L\$$  is the regular language of arrows
  of a non-deterministic finite state automaton over $\Monoid[\Sigma]^\top$.
\end{proposition}

\subsection{Non-deterministic tree automata as finitary ULF functors over operads}\label{section/tree-automata}
One nice aspect of the fibrational approach to non-deterministic finite state automata
based on finitary ULF functors is that it adapts smoothly when one shifts from word automata to tree automata.
As a first step in that direction, we first describe how the ULF and finite fiber properties may be extended to functors of operads.
\begin{definition}\label{definition/ULF}
  A functor of operads $p:\Der\to\Base$ has the \defin{unique lifting of factorizations property} (or \defin{is ULF}) if any of the following equivalent conditions hold:
  \begin{enumerate}
  \item
    for any operation $\alpha$ of $\Der$, if $p(\alpha) = g \circ (h_1,\dots,h_n)$ for some operation $g$ and list of operations $h_1,\dots,h_n$ of $\Base$, there exists a unique operation $\beta$ and list of operations $\gamma_1,\dots,\gamma_n$ of $\Der$ such that $\alpha = \beta \circ (\gamma_1,\dots,\gamma_n)$ and $p(\beta) = g$ and $p(\gamma_1) = h_1, \dots, p(\gamma_n) = h_n$;
  \item
    for any operation $\alpha$ of $\Der$, if $p(\alpha) = g \circ_i h$ for some operations $g$ and $h$ of $\Base$ and index $i$, there exists a unique pair of operations $\beta$ and $\gamma$ of $\Der$ such that $\alpha = \beta \circ_i \gamma$ and $p(\beta) = g$ and $p(\gamma) = h$;
  \item
    the structure maps of the associated lax functor of operads $F:\Base\to\Span{\Set}$ discussed in \S\ref{section/benabou} are invertible.
  \end{enumerate}
\end{definition}
\begin{definition}
  We say that a functor of operads $p : \Qoperad \to \Ooperad$ is \defin{finitary} if either of the following equivalent conditions hold:
  \begin{itemize}
  \item the fiber $p^{-1}(A)$ as well as the fiber $p^{-1}(f)$ is finite for every color $A$ and operation $f$ of the operad $\Ooperad$;
  \item the associated lax functor of operads $F:\Ooperad\to\Span{\Set}$ factors via $\Span{\FinSet}$.
  \end{itemize}
\end{definition}
\noindent
One can check that the underlying bare automaton $M=(\Sigma,Q,\TranFunc)$ of any non-deterministic finite state tree automaton \cite{tata2008} 
gives rise to a finitary ULF functor of operads $p:\Qoperad\to\FreeOperad{\Sigma}$,
where $\FreeOperad{\Sigma}$ is the free operad generated by the ranked alphabet~$\Sigma$ (which may be seen as an uncolored non-symmetric species),
where the operad~$\Qoperad$ has states of the automaton as colors, and operations freely generated by $n$-ary nodes of the form $t : q_1,\dots,q_n \to q$ for every transition $t\in\TranSet$ of the form $\TranFunc(t)=(q_1,\dots,q_n,a,q)$
where $a$ is an $n$-ary letter in $\Sigma$,
and where $p$ transports every such $n$-ary transition $t : q_1,\dots,q_n \to q$ to the underlying
$n$-ary letter $a:\ast,\dots,\ast\to\ast$.
This motivates us to proceed as for word automata and propose a more general notion 
of finite state automaton over an arbitrary operad:

\begin{definition}\label{definition/tree-automaton}
  A \defin{non-deterministic finite state automaton over an operad} is given by a tuple
  $M = (\Ooperad, \Qoperad, p : \Qoperad \to \Ooperad, q)$
  consisting of two operads $\Ooperad$ and $\Qoperad$, a finitary ULF functor of operads $p:\Qoperad\to\Ooperad$,
  and a color $q$ of $\Qoperad$.
  A color of~$\Qoperad$ is called a \defin{state}, and an operation of~$\Qoperad$ is called a \defin{run tree} of the automaton~$p:\Qoperad\to\Ooperad$.
  The \defin{regular language of constants} $\Lang{M}$ recognized by the automaton
  is the set of constants $c$ in $\Ooperad$ that can be lifted along~$p$ to a constant $\alpha:q$ in $\Qoperad$, that is $\Lang{M} = \set{p(\alpha) \mid \alpha : q}$.
\end{definition}
\noindent
%
%

\subsection{From a word automaton to a tree automaton on spliced words}
We now state a simple property of ULF functors establishing a useful connection
between word and tree automata.
\begin{proposition}\label{proposition/Words-ULF}
Suppose that $p:\Qcategory\to\Ccategory$ is a functor of categories.
Then, if $p$ is ULF functor, so is the functor of operads~$\Words{p}:\Words{\Qcategory}\to\Words{\Ccategory}$.
Moreover, if $p$ is finitary then so is $\Words{p}$.
\end{proposition}
\noindent
From this it follows that every non-deterministic finite state automaton $M = (\Ccategory,\Qcategory,p,q_0,q_f)$ over a given category $\Ccategory$ induces a non-deterministic finite state automaton $\Words{M} = (\Words{\Ccategory},\Words{\Qcategory},\Words{p},(q_0,q_f))$ over the spliced arrow operad~$\Words{\Ccategory}$.
%
%
Moreover, it is immediate that $\Lang{M}=\Lang{\Words{M}}$ since the constants of $\Words{\Ccategory}$ are exactly the arrows of $\Ccategory$.
As we will see, these observations play a central role in our understanding of the Chomsky and Schützenberger
representation theorem~\cite{ChomskySchuetzenberger1963}.
Finally, let us emphasize that the notion of finite state automaton over an operad
is really a proper generalisation of the classical notion of tree automaton since it allows taking a
non-free operad as target, and in particular the non-free operad of sliced arrows (see Remark~\ref{remark/non-free}).
This is what enables us to transform an automaton on the arrows of $\Ccategory$ into an automaton on the operations of the spliced arrow operad $\Words{\Ccategory}$, which could be seen as a kind of tree automaton over ``trees that bend'' (cf.~Fig.~\ref{fig:spliced-arrows}).

\section{The Chomsky-Schützenberger Representation Theorem}\label{section/CSR-theorem}



\subsection{Pulling back context-free grammars along finite state automata}\label{section/pulling-back}
\begin{proposition}\label{proposition/pullback-along-ULF}
Suppose given a species~$\Sspecies$, a functor of operads ${p:\FreeOperad{\Sspecies}\to\Ooperad}$
and a ULF functor of operads $p_{\Qoperad}:\Qoperad\to\Ooperad$.
In that case, the pullback of $p$ along $p_{\Qoperad}$ in the category of operads is obtained from
a corresponding pullback of $\phi : \Sspecies \to \ForgetOperad{\Ooperad}$ along $\ForgetOperad{p_{\Qoperad}} : \ForgetOperad{\Qoperad}\to\ForgetOperad{\Ooperad}$ in the category of species:
\begin{equation}\label{equation/pullback-of-operads-species}
\begin{tikzcd}[column sep=1em,row sep=.4em]
\FreeOperad{\Sspecies'}\arrow[dd,"{p'}"{swap}]\arrow[rr,"\FreeOperad{\psi}"] && \FreeOperad{\Sspecies}\arrow[dd,"p"]
\\
& pullback &
\\
\Qoperad\arrow[rr,"{p_{\Qoperad}}"{swap}] && \Ooperad
\end{tikzcd}
\qquad
\begin{tikzcd}[column sep=2em,row sep=.4em]
{\Sspecies'}\arrow[dd,"{\phi'}"{swap}]\arrow[rr,"\psi"] && {\Sspecies}\arrow[dd,"\phi"]
\\
& pullback &
\\
{\ForgetOperad{\Qoperad}}\arrow[rr,"{\ForgetOperad{p_{\Qoperad}}}"{swap}] && {\ForgetOperad{\Ooperad}}
\end{tikzcd}
\end{equation}
\end{proposition}
\noindent
This observation may be applied to pull back a context-free grammar~$G=(\Ccategory,\Sspecies,S,p)$ of arrows in a category~$\Ccategory$,
along a non-deterministic finite state automaton $M = (\Ccategory,\Qcategory,p_M:\Qcategory_M\to\Ccategory,q_0,q_f)$ over the same category.
The construction is performed by first
considering the ULF functor of operads of spliced arrows
${\Words{p_M}} : \Words{\Qcategory}\to \Words{\Ccategory}$
deduced from the ULF functor of categories~$p_M$ using Prop.~\ref{proposition/Words-ULF}.
We therefore have a pullback diagram of the form~\eqref{equation/pullback-of-operads-species}
in the category of operads
for $p_\Qoperad = \Words{p_M}$
where the species~$\Sspecies'$ and the map of species~$\phi' : \Sspecies' \to \ForgetOperad{\Words{\Qcategory}}$ determining $p'$
are computed by a pullback in the category of species.
This pullback admits a concrete description:
the colors of~$\Sspecies'$ are triples $(q,R,q')$ where $p(R)=(p_M(q),p_M(q'))$
and its $n$-ary nodes $(q_1,R_1,q_1'),\dots,(q_n,R_n,q_n')\to(q,R,q')$
are pairs $(x,\alpha)$ of a $n$-ary node $x: R_1,\dots,R_n\to R$ of the species~$\Sspecies_G$
together with a $n$-ary spliced arrow $\alpha={\alpha_0}{-}\dots{-}{\alpha_n}:(q_1,q_1'),...,(q_n,q_n') \to (q,q')$
in~$\Words{\Qcategory}$ such that $p_M(\alpha)=p(x)$, while the map of species~$\phi'$ transports
a color $(q,R,q')$ to the color $(q,q')$ of~$\Words{\Qcategory}$ and a $n$-ary node $(x,\alpha)$ 
to the $n$-ary operation~$\alpha$.
Since $\Sspecies$ is finite and the functor $p_M$ has finite fibers, the species $\Sspecies'$ is also finite.
To complete the construction, the pullback grammar $G'=(\Qcategory,\Sspecies',S',p')$ 
is defined by taking the color $S'=(q_0,S,q_f)$ of the species~$\Sspecies'$ as start symbol.
Note that $G'$ is a context-free grammar over the arrows of $\Qcategory_M$, which correspond to runs of the automaton~$M$.
In traditional syntax of context-free grammars, we could describe it as having a production rule
$(q,R,q')\to\alpha_0(q_1,R_1,q'_1)\alpha_1\dots (q_n,R_n,q'_n)\alpha_n$
for every production rule $R\to w_0R_1w_1\dots R_n w_n$ of the original grammar~$G$ 
and sequence of $n+1$ runs of the automaton $\alpha_0:q\to q_1$, $\alpha_1:q'_1\to q_2$, $\dots$, $\alpha_n:q'_{n}\to q'$ over the respective words $w_0,\dots,w_n$.

We can then also derive a grammar~$G'' = p_M(G')$ of arrows in $\Ccategory$ by taking 
the functorial image (Prop.~\ref{proposition/basic-closure}(iii)) of~$G'$ along the functor~$p_M:\Qcategory\to\Ccategory$, which by construction will generate the intersection of the context-free language of $G$ and the regular language of $M$.
\begin{proposition}\label{proposition/intersection-of-languages}
For $G'$ and $G''$ defined as above, we have $\Lang{G'} = p_M^{-1}(\Lang{G}) \cap \Qcategory(q_0,q_f)$ and $\Lang{G''} = \Lang{G} \cap \Lang{M}$.
\end{proposition}
\begin{corollary}
  Context-free languages of arrows are closed under pullback along non-deterministic finite state automata,
  and under intersection with regular languages.
\end{corollary}
\begin{example}\label{example/parsing-intervals}
  For any word $w = a_1 \dots a_n$ of length $n$, there is an $(n+1)$-state automaton $M_w$ that recognizes the singleton language $\set{w}$, with initial state $0$, accepting state $n$, and transitions of the form $(i,a_{i+1},i+1)$ for each $0 \le i < n$.
  By pulling back any context-free grammar $G$ along $M_w$, we obtain a new grammar that may be seen as a specialization of $G$ to the word $w$, with non-terminals $(i,R,j)$ representing the fact that the subword $w_{i,j} = a_{i+1}\dots a_j$ parses as $R$ (cf.~\S\ref{section/application-to-parsing}).
  This example generalizes to context-free grammars of arrows over any category $\Ccategory$ with the property that every arrow $w$ has only finitely many factorizations $w = uv$ of length 2, by observing that the underlying bare automaton of $M_w$ is isomorphic to the \emph{interval category} \cite{LawvereMenni2010} of $w$.
  In general, for any arrow $w : A \to B$ of a category $\Ccategory$, the interval category $\interval{w}$ is defined by taking objects to be triples $(X,u,v)$ of an object $X \in \Ccategory$ and a pair of arrows $u : A \to X$, $v : X \to B$ such that $w = uv$, and arrows  $(X,u,v) \to (X',u',v')$ to be arrows $x : X \to X'$ such that $u' = ux$ and $v = xv'$.
  The interval category $\interval{w}$ has an initial object $(\id[A],w)$ and a terminal object $(w,\id[B])$, and it comes equipped with an evident forgetful functor $\interval{w} \to \Ccategory$, which is always ULF, and moreover finitary by the stated condition on $\Ccategory$.
  The tuple $M_w = (\Ccategory,\interval{w}, \interval{w} \to \Ccategory, (\id[A],w), (w,\id[B]))$ therefore defines a finite-state automaton, and any CFG of arrows over $\Ccategory$ can be pulled back along $M_w$ to obtain a CFG specialized to the arrow $w$.
\end{example}

\subsection{The contour category of an operad and the contour / splicing adjunction}
\label{section/contour-category}

In \S\ref{section/spliced-arrows}, we saw how to construct a functor
\[ \Words{-} : \Cat \to \Operads{} \]
transforming any category $\Ccategory$ into an operad $\Words{\Ccategory}$ of spliced arrows of arbitrary arity, which played a central role in our definition of context-free language of arrows in a category.
We construct now a left adjoint functor
\[
  \Contour{-} : \Operads{} \to \Cat
\]
which extracts from any given operad $\Ooperad$ a category $\Contour{\Ooperad}$ whose arrows correspond to ``oriented contours'' along the boundary of the operations of the operad.

\begin{definition}\label{definition/contour-category}
The \defin{contour category $\Contour{\Ooperad}$} of an operad $\Ooperad$ is defined as a quotient of the following free category:
\begin{itemize}
\item objects are given by \emph{oriented colors} $R^\epsilon$ consisting of a color $R$ of $\Ooperad$ and an orientation $\epsilon \in \set{\tagU,\tagD}$ (``up'' or ``down'');
\item arrows are generated by pairs $(f,i)$ of an operation $f : R_1,\dots,R_n \to R$ of $\Ooperad$ and an index $0 \le i \le n$, defining an arrow $R_i^\tagD \to R_{i+1}^\tagU$ under the conventions that $R_0^\tagD = R^\tagU$ and $R_{n+1}^\tagU = R^\tagD$;
\end{itemize}
subject to the conditions that $\id[R^\tagU] = (\id[R], 0)$ and $\id[R^\tagD] = (\id[R], 1)$ as well as the following equations:
\begin{align}
  (f \circ_i g, j)  &= \begin{cases}
                         (f,j) & j < i  \\ 
                         (f,i)(g,0)\hphantom{(f,i+1)} & j = i \\
                         (g,j-i) & i < j < i+m \\
                         (g,m)(f,i+1) & j = i+m \\
                         (f,j-m+1) & j > i+m
                       \end{cases} \label{equation/contour1}\\ 
  (f \circ_i c, j) &= \begin{cases}
                         (f,j) & j < i \\
                         (f,i)(c,0)(f,i+1) & j = i \\
                         (f,j+1) & j > i
                       \end{cases} \label{equation/contour2}
\end{align}
whenever the left-hand side is well-formed,
for every operation $f$, operation $g$ of positive arity $m > 0$, constant $c$,
and indices $i$ and $j$ in the appropriate range.

We refer to each generating arrow $(f,i)$ of the contour category $\Contour{\Ooperad}$ as a \defin{sector} of the operation~$f$.
See Fig.~\ref{fig:contour-category} for a graphical interpretation of sectors and of the equations on contours seen as compositions of sectors.
\end{definition}
\begin{figure}
  \begin{center}\includegraphics[width=\textwidth]{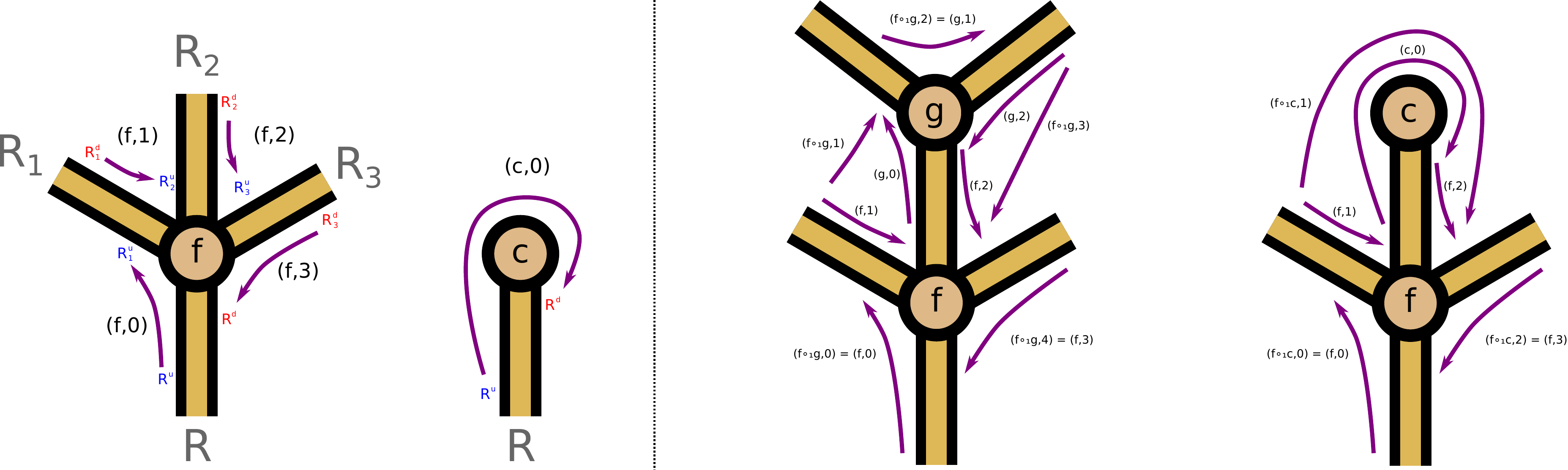}\end{center}
  \caption{Left: interpretation of the generating arrows of the contour category $\Contour{\Ooperad}$. Right: interpretation of equations \eqref{equation/contour1} and \eqref{equation/contour2}.}
  \label{fig:contour-category}
\end{figure}
\begin{remark}
  In the case of a free operad over a species $\Sspecies$, we also write $\Contour{\Sspecies}$ for the contour category $\Contour{\FreeOperad{\Sspecies}}$
  because it admits an even simpler description as a free category generated by the arrows $(x,i) : R_i^\tagD \to R_{i+1}^\tagU$ for every node $x : R_1,\dots,R_n \to R$ of the species $\Sspecies$.
  We refer to these generating arrows $(x,i)$ consisting of an $n$-ary node $x$ and an index $0 \le i \le n$ as \defin{corners} since they correspond to the corners 
  of $\Sspecies$-rooted trees seen as rooted planar maps \cite{Schaeffer2015map}.
  Note that every sector of an operation of $\FreeOperad{\Sspecies}$ factors uniquely in the contour category~$\Contour{\Sspecies}$ as a sequence of corners.
  Thinking of the nodes of $\Sspecies$ as the production rules of a context-free grammar, the corners $(x,i)$ correspond exactly to the \emph{items} used in LR parsing and Earley parsing \cite{Knuth1965,Earley1970}.
\end{remark}

The contour construction provides a left adjoint to the spliced arrow construction
because a functor of operads
$\Ooperad \to \Words{\Ccategory}$
is entirely described by the data of a pair of objects $(A,B) = (R^\tagU, R^\tagD)$ in $\Ccategory$ for every color $R$ in $\Ooperad$ together with a sequence $f_0, f_1, \dots, f_n$ of $n+1$ arrows in $\Ccategory$, where $f_i : R_i^\tagD \to R_{i+1}^\tagU$ for $0 \le i \le n$ for each operation $f : R_1,\dots,R_n \to R$ of $\Ooperad$, under the same conventions as above.
The equations \eqref{equation/contour1} and \eqref{equation/contour2} on the generators of $\Contour{\Ooperad}$ reflect the equations imposed by the functor of operads $\Ooperad \to \Words{\Ccategory}$ on the spliced arrows of $\Ccategory$ appearing as the image of operations in $\Ooperad$.
In that way we transform any functor of operads $\Ooperad \to \Words{\Ccategory}$ into a functor
$\Contour{\Ooperad} \to \Ccategory$
which may be seen as an interpretion of the contours of the operations of $\Ooperad$ in $\Ccategory$.

The unit and counit of the contour / splicing adjunction also have nice descriptions.
The unit of the adjunction defines, for any operad $\Ooperad$, a functor of operads $\Ooperad \to \Words{\Contour{\Ooperad}}$ that acts on colors by $R \mapsto (R^\tagU, R^\tagD)$,
and on operations by sending an operation $f : R_1,\dots,R_n \to R$ of $\Ooperad$ to the spliced word of sectors $(f,0){-}\dots{-}(f,n) : (R_1^\tagU,R_1^\tagD),\dots,(R_n^\tagU,R_n^\tagD) \to (R^\tagU,R^\tagD)$.
The counit of the adjunction defines, for any category $\Ccategory$, a functor of categories $\Contour{\Words{\Ccategory}} \to \Ccategory$ that acts on objects by $(A,B)^\tagU \mapsto A$ and $(A,B)^\tagD \mapsto B$, and on arrows by sending the $i$th sector of a spliced word to its $i$th word, $(w_0{-}\dots{-}w_n,i) \mapsto w_i$.


In contrast to the situation of Prop.~\ref{proposition/Words-ULF}, it is \emph{not} the case that $\Contour{-}$ always preserves the ULF property.
\begin{remark}
Consider the category $\mathbf{2}$ with two objects $A$ and $B$ and only identity arrows,
and the unique functor $p$ to the terminal category $\mathbf{1}$.
We claim that the associated ULF functor of operads $\Words{p}$ induces a functor of categories $\Contour{\Words{p}}$ which is not ULF.
%
Consider the two binary operations $f = \id[A]{-}\id[A]{-}\id[A]$ and $g = \id[A]{-}\id[A]{-}\id[B]$
and the constant $c = \id[A]$ in $\Words{\mathbf{2}}$,
as well as the binary operations $h=\id[\ast]{-}\id[\ast]{-}\id[\ast]$ and the constant $d=\id[\ast]$
in $\Words{\mathbf{1}}$.
The category~$\Contour{\Words{{\mathbf{2}}}}$ has the sequence of sectors $\alpha = (f,0)(c,0)(g,1)$ as an arrow, which
is different from the identity. 
On the other hand, it is mapped by $\Contour{\Words{p}}$ to the sequence $w = (h,0)(d,0)(h,1)$, which is equal thanks to Equation~\eqref{equation/contour2}
to the sector $(h\circ_0 d,0)$ of the unary operation~$h\circ_0 d = \id[\ast]{-}\id[\ast]$ of $\Words{\mathbf{1}}$,
and hence $w = \id[(\ast,\ast)^\tagU]$.
Since the factorization $\id = \id \id$ in $\Words{\mathbf{1}}$ lifts to two distinct factorizations $\alpha = \id \alpha = \alpha \id$ in $\Words{\mathbf{2}}$, $p$ is not ULF.
\end{remark}
\noindent
Still, we can verify that maps of species induce ULF functors between their contour categories.
\begin{proposition}\label{proposition/Contour-species-ULF}
If $\psi:\Sspecies\to \Rspecies$ is a map of species, then $\Contour{p} : \Contour{\Sspecies} \to \Contour{\Rspecies}$ is a ULF functor of categories.
\end{proposition}



\subsection{The universal context-free grammar of a pointed species, and its associated tree contour language}
\label{section/unigram}

Every finite species~$\Sspecies$ equipped with a color~$S$ 
comes with a \emph{universal} context-free grammar~$\UnivGrammar{\Sspecies,S}=(\Contour{{\Sspecies}},\Sspecies,S,p_\Sspecies)$, 
characterized by the fact that $p_\Sspecies:{\FreeOperad{\Sspecies}}\to{\Words{\Contour{{\Sspecies}}}}$
is the unit of the contour / splicing adjunction.
By ``universal'' context-free grammar, we mean that any context-free grammar of arrows
$G=(\Ccategory,\Sspecies,S,p)$ with the same underlying species and start symbol factors uniquely through~$\UnivGrammar{\Sspecies,S}$ in the sense
that there exists a unique functor $q_G:\Contour{{\Sspecies}}\to\Ccategory$ satisfying the equation
\begin{center}
\begin{tikzcd}[column sep=2em]
{\FreeOperad{\Sspecies}}
\arrow[rr,"{p}"]
&&
\Words{\Ccategory}
\quad = \quad
{\FreeOperad{\Sspecies}}
\arrow[rr,"{p_{\Sspecies}}"] 
&& 
{\Words{\Contour{{\Sspecies}}}}
\arrow[rr,"{\Words{q_G}}"] && \Words{\Ccategory}
\end{tikzcd}
\end{center}
We refer to the language of arrows~$\Lang{\UnivGrammar{\Sspecies,S}}$, also noted $\Lang{\Sspecies,S}$, as a \defin{tree contour language}, and to its arrows as \defin{tree contour words}, since they describe the contours of $\Sspecies$-rooted trees with root color $S$, see left side of Fig.~\ref{fig:contour-word} for an illustration.
The factorization above shows that the context-free grammar $G$ is the functorial image of the universal grammar $\UnivGrammar{\Sspecies,S}$ along the functor of categories~$q_G$, whose purpose is to transport each corner of a node in $\Sspecies$ to the corresponding arrow in~$\Ccategory$ as determined by the grammar~$G$.
At the level of languages, we have 
$\Lang{G} = q_G\,\Lang{\Sspecies,S}$.

%
\begin{figure}
  \begin{center}\includegraphics[width=0.35\textwidth]{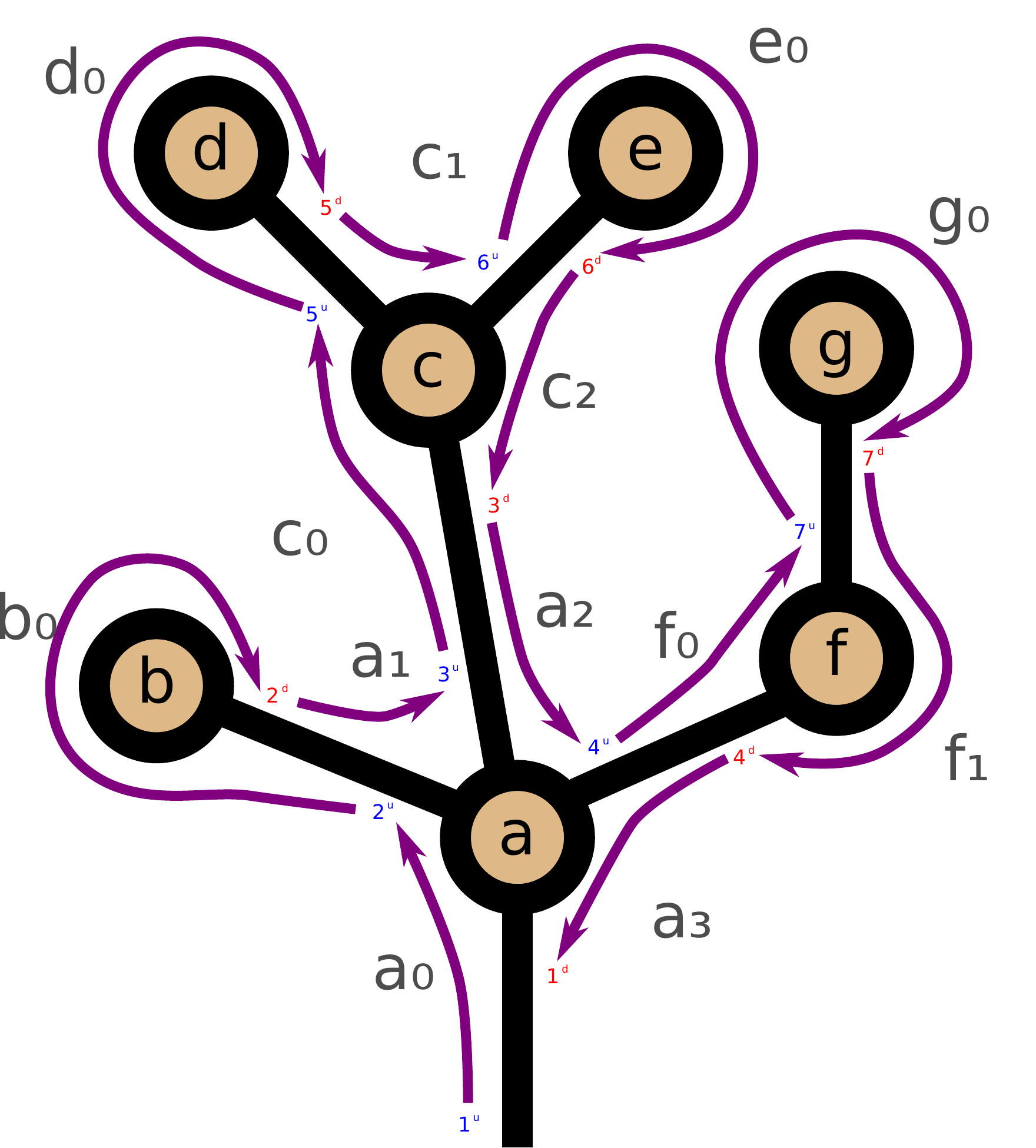} \qquad \vrule \qquad \includegraphics[width=0.35\textwidth]{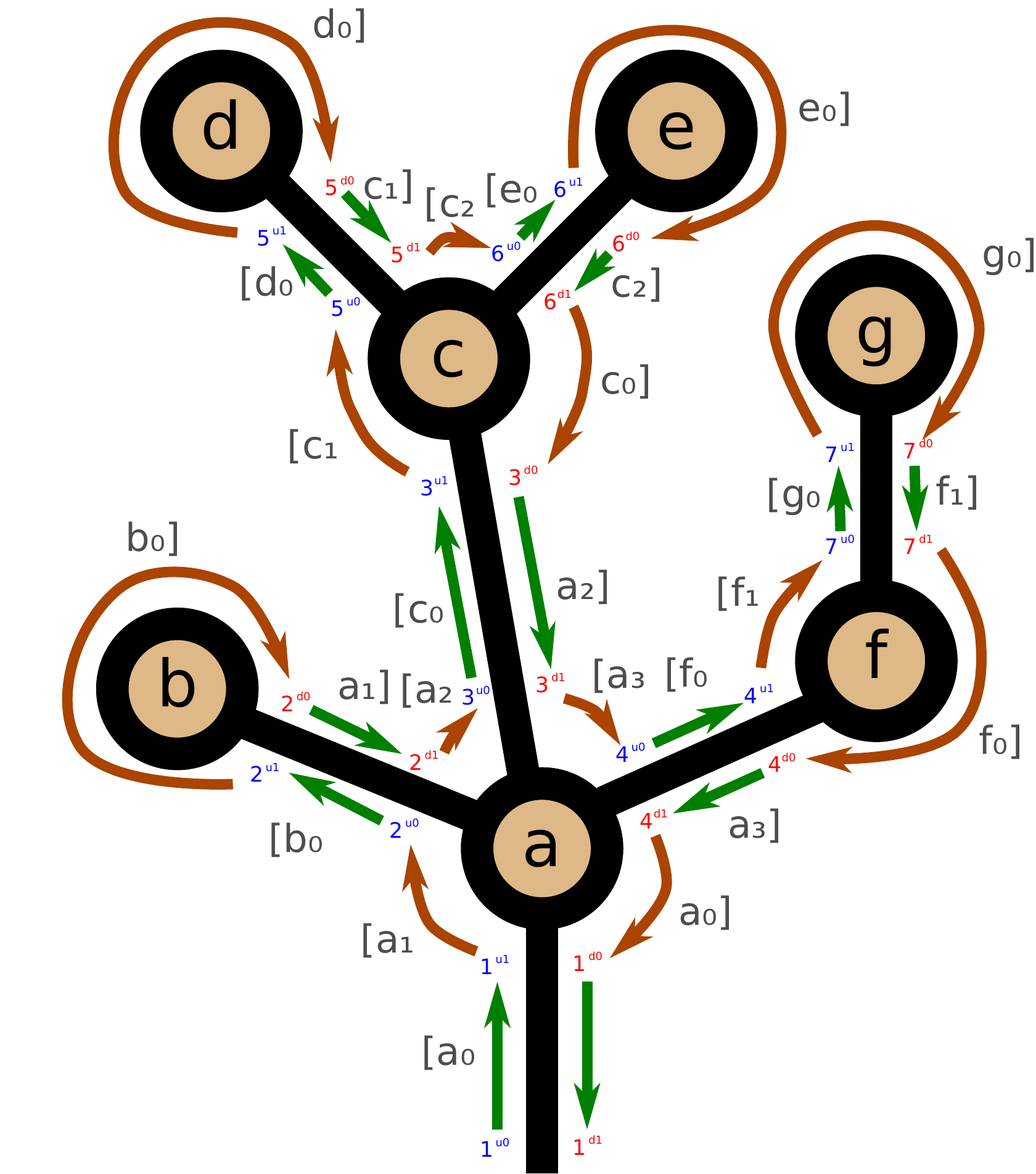} \end{center}
  \caption{Left: an $\Sspecies$-rooted tree of root color 1 and its corresponding contour word $\mathsf{a_0 b_0 a_1 c_0 d_0 c_1 e_0 c_2 a_2 f_0 g_0 f_1 a_3} : 1^\tagU \to 1^\tagD$.
    Right: the corresponding Dyck word obtained by first decomposing each corner of the contour into alternating actions of walking along an edge and turning around a node, and then annotating each arrow both by the orientation (with $u = \mathsf{[}, d = \mathsf{]}$) and the node-edge pair of its target.
  }
  \label{fig:contour-word}
\end{figure}


\begin{remark}
  The notion of tree contour language makes sense even for non-finitary pointed species $(\Sspecies,S)$, although in that case the resulting universal grammar $\UnivGrammar{\Sspecies,S}$ is no longer context-free, having infinitely many non-terminals.
  Still, it may be an interesting object of study.
  In particular, the tree contour language $\UnivGrammar{\TerminalSpecies,\ast}$ generated by the \emph{terminal species} $\TerminalSpecies$ with one color and a single operation of every arity appears to be of great combinatorial interest, with words in the language describing the shapes of rooted planar trees with arbitrary node degrees.
\end{remark}

\subsection{Representation theorem}
\label{section/represenation-theorem}

The achievement of the Chomsky-Sch{\"u}tzenberger representation theorem \cite[\S5]{ChomskySchuetzenberger1963} is to separate any
context-free grammar $G=(\Sigma,\Nonterminals,\Sentence,\Productions)$
into two independent components: a context-free grammar $\Dyck{n}$ with only one non-terminal over an alphabet $\Sigma_{2n} = \set{ [_1, ]_1, \dots, [_n, ]_n}$ of size $2n$ (for some $n$), which generates Dyck words of balanced brackets describing the shapes of parse trees with nodes labelled by
production rules of $G$ ;
and a finite state automaton $M$ to check that the edges of these trees may be appropriately colored by
the non-terminals of $G$ according to the labels of the nodes specifying the productions.
The original context-free language $\Lang{G}$ is then obtained as the image of the intersection of the
languages generated by $\Dyck{n}$ and by $M$, under a homomorphism $\Sigma_{2n}^* \to \Sigma^*$ that interprets
each bracket of the Dyck word by a word in the original alphabet, with a choice to either interpret the open or the close brackets as empty words.

In this section, we give a new proof of the Chomsky-Sch{\"u}tzenberger representation theorem,
generalized to context-free grammars of arrows $G$ over any category~$\Ccategory$.
Since the category $\Ccategory$ may have more than one object, the appropriate statement of the representation theorem cannot require the grammar describing the shapes of parse trees to have only one non-terminal, but we can nonetheless construct one that is \emph{$\Ccategory$-chromatic} in the following sense.
\begin{definition}
A context-free grammar of arrows in $\Ccategory$ is \defin{$\Ccategory$}-chromatic
when its non-terminals are the colors of~$\Words{\Ccategory}$, in other words the pairs~$(A,B)$ of objects of the category~$\Ccategory$.
\end{definition}
\noindent
Moreover, rather than using Dyck words to represent parse trees, we find it more natural to use tree contour words, based on the observation given above
(\S\ref{section/unigram}) that every context-free language may be \emph{canonically} represented as the image of a tree contour language generated by a context-free grammar with the same set of non-terminals.

As preparation to our proof of the representation theorem, we establish:
\begin{proposition}\label{proposition/node-colors-factorization}
  Let $\phi:\Sspecies \to \Rspecies$ be a map of species with underlying change-of-color function $\phi_C$.
  Let $\phi_C\,\Sspecies$ be the species with the same underlying set of nodes as~$\Sspecies$, but where every node $x:R_1,\dots,R_n\to R$ in~$\Sspecies$ becomes a node $x:\phi_C(R_1),\dots,\phi_C(R_n)\to \phi_C(R)$ in~$\phi_C\,\Sspecies$.
  Then $\phi$ factors as 
\begin{center}
\begin{tikzcd}[column sep=2em]
\Sspecies
\arrow[rr,"{\phi}"]
&&
\Rspecies
\quad = \quad
\Sspecies
\arrow[rr,"{\phi_{\colors}}"] 
&& 
\phi_C\,\Sspecies
\arrow[rr,"{\phi_\nodes}"] &&
\Rspecies
\end{tikzcd}
\end{center}
where $\phi_\colors$ is the identity on nodes and $\phi_\nodes$ is the identity on colors.
\end{proposition}

\begin{proposition}\label{proposition/pullback-sliced-contour}
Every map of species $\psi:\Sspecies\to\Sspecies'$ injective on nodes
induces a commutative diagram 
\begin{equation}\label{equation/pullback-along-injective}
\begin{tikzcd}[column sep=3em,row sep=1em]
{\FreeOperad{\Sspecies}}
\arrow[dd,"{p_{\Sspecies}}"{swap}]\arrow[rr,"{\FreeOperad{\psi}}"] 
&& 
{\FreeOperad{\Sspecies'}}
\arrow[dd,"{p_{\Sspecies'}}"]
\\
&  &
\\
{\Words{\Contour{{\Sspecies}}}}
\arrow[rr,"{\Words{\Contour{{\psi}}}}"{swap}] &&
{\Words{\Contour{{\Sspecies'}}}}
\end{tikzcd}
\end{equation}
where the canonical functor of operads from $\FreeOperad{\Sspecies}$ to the pullback of $p_{\Sspecies'}$ along $\Words{\Contour{{\psi}}}$ is fully faithful.
\end{proposition}
\noindent
Now, let $G=(\Ccategory,\Sspecies,S,p)$ be any context-free grammar of arrows, and by Prop.~\ref{proposition/node-colors-factorization} consider the corresponding $\Ccategory$-chromatic grammar $G_\nodes = (\Ccategory,\phi_C\,\Sspecies,(A,B),p_\nodes)$, where $(A,B) = p(S)$.
We have a commutative diagram
\begin{equation}\label{equation/pullback-of-C-chromatic}
\begin{tikzcd}[column sep=3em,row sep=1em]
{\FreeOperad{\Sspecies}}
\arrow[dd,"{p_{\Sspecies}}"{swap}]\arrow[rr,"{\FreeOperad{\phi_\colors}}"]
&& 
{\FreeOperad{\phi_C\,\Sspecies}}
\arrow[dd,"{p_{\phi_C\,\Sspecies}}"]
\\
&  &
\\
{\Words{\Contour{{\Sspecies}}}}
\arrow[rr,"{\Words{\Contour{{\phi_\colors}}}}"{swap}] 
\arrow[rd,"{\Words{q_{G}}}"{swap}] 
&&
{\Words{\Contour{{\phi_C\,\Sspecies}}}}
\arrow[ld,"{\Words{q_{G_\nodes}}}"] 
\\
&
\Words{\Ccategory}
\end{tikzcd}
\end{equation}
where the commutativity of the lower triangle follows from the equation~$\phi = \phi_\colors\phi_\nodes$
and the contour / splicing adjunction.
Note also that $\Contour{\phi_\colors}$ is a ULF functor of categories by Prop.~\ref{proposition/Contour-species-ULF} and also finitary
because $\phi_\colors$ is finitary in the expected sense (and even finite).
From this follows that $M_\colors = (\Contour{\phi_C\,\Sspecies},\Contour{\Sspecies},\Contour{\phi_\colors},S^\tagU,S^\tagD)$ defines a finite-state automaton.
By Props~\ref{proposition/intersection-of-languages} and \ref{proposition/pullback-sliced-contour} and the translation principle,
we deduce that
%
$$
\Contour{\phi_\colors} \, \Lang{\Sspecies,S} \, = \, \Lang{\phi_C\,\Sspecies,(A,B)} \, \cap \, L_{M_\colors}.
$$
Finally, using that $G$ is the image of the universal grammar $\UnivGrammar{\Sspecies,S}$ and considering the commutative diagram~\eqref{equation/pullback-of-C-chromatic}, we conclude:
\[
  L_G = q_G\,L_{\Sspecies,S} = q_{G_\nodes}\,\Contour{\phi_\colors} \, L_{\Sspecies,S} = q_{G_\nodes}\, (\Lang{\phi_C\,\Sspecies,(A,B)} \, \cap \, L_{M_\colors}).
\]

\begin{theorem}\label{theorem/CS}
Every context-free language of arrows of a category~$\Ccategory$
is the functorial image of the intersection of a $\Ccategory$-chromatic context-free tree contour language with a regular language.
\end{theorem}
\noindent
The original statement of the Chomsky-Sch{\"u}tzenberger theorem can be recovered by relying
on the fact that any tree contour word can be faithfully translated to a Dyck word, via an easy translation that doubles the number of letters,
and which also has a geometric interpretation that involves decomposing each corner of the contour into alternating actions of walking along an edge and turning around a node, see right side of Fig.~\ref{fig:contour-word}.
Intriguingly, this decomposition suggests the existence of an embedding of the contour category $\Contour{\Sspecies}$ into a \emph{bipartite} contour category, where each object $R^\epsilon$ has been replaced by a pair of objects $R^{\epsilon0}$ and $R^{\epsilon1}$, in a way that is analogous to the embedding of the ``oriented cartographic group'' used to represent maps on oriented surfaces into the cartographic group for maps on not necessarily orientable surfaces (cf.~\cite{ShabatVoevodsky1990,JonesSingerman1994}).

\ack{We thank Bryce Clarke for helpful discussions about this work, and to the anonymous reviewers for comments improving the presentation.}

\bibliographystyle{entics}
\bibliography{refs}

\ifwithAppendices

\appendix

\section{Supplementary proofs}

\subsection*{Proofs of Props.~\ref{proposition/basic-closure}(i) and (ii).}

\begin{enumerate}
\item[(i)] 
  Given two grammars $G_1 = (\Ccategory, \Sspecies_1, S_1, p_1)$ and $G_2 = (\Ccategory, \Sspecies_2, S_2, p_2)$, where $S_1$ and $S_2$ both refine the same gap type $(A,B)$, we define a new grammar $G = (\Ccategory, \Sspecies, S, p)$ that generates the union of the two languages $\Lang{G} = \Lang{G_1} \cup \Lang{G_2}$ by taking $\Sspecies$ to be the disjoint union of the colors and operations of $\Sspecies_1$ and $\Sspecies_2$ combined with a distinguished color $S$ and pair of unary nodes $i_1 : S_1 \to S$ and $i_2 : S_2 \to S$, and defining $\phi : \Sspecies \to \ForgetOperad{\Words{\Ccategory}}$ to be the copairing of $\phi_1$ and $\phi_2$ extended with the mappings $\phi(S) = (A,B)$ and $\phi(i_1) = \phi(i_2) = \id[A]{-}\id[B]$.
\item[(ii)]
  Given grammars $G_1 = (\Ccategory,\Sspecies_1,S_1,p_1)$, \dots, $G_n = (\Ccategory,\Sspecies_n,S_n,p_n)$ where $S_i \refs (A_i,B_i)$ for each $1 \le i \le n$, together with an operation
  $w_0{-}w_1{-}\dots{-}w_n : (A_1,B_1),\dots,(A_n,B_n) \to (A,B)$ of $\Words{\Ccategory}$, we construct a new grammar $G = (\Ccategory,\Sspecies,S,p)$ that generates the spliced concatenation $w_0 \Lang{G_1} w_1 \dots \Lang{G_n} w_n$ by taking $\Sspecies$ to be the disjoint union of the colors and operations of $\Sspecies_1,\dots,\Sspecies_n$ combined with a distinguished color $S$ and a single $n$-ary node $x : S_1,\dots,S_n \to S$, and defining $\phi : \Sspecies \to \ForgetOperad{\Words{\Ccategory}}$ to be the cotupling of $\phi_1,\dots,\phi_n$ extended with the mappings $\phi(S) = (A,B)$ and $\phi(x) = w_0{-}w_1{-}\dots{-}w_n$.
\end{enumerate}

\subsection*{Proof of Prop.~\ref{proposition/bilinear-normal-form}}

Let $G = (\Ccategory,\Sspecies,S,p)$ be a context-free grammar of arrows.
A bilinear grammar $G_\bin = (\Ccategory,\Sspecies_\bin,S,p_\bin)$ over the same category and with the same start symbol is constructed as follows.
$\Sspecies_\bin$ includes all of the colors and all of the nullary nodes of $\Sspecies$, with $\phi_\bin(R) = \phi(R)$ and $\phi_\bin(c) = \phi(c)$.
Additionally, for every node $x : R_1,\dots,R_n \to R$ of $\Sspecies$ of positive arity $n>0$, where $\phi(x) = w_0{-}\dots{-}w_n : (A_1,B_1),\dots,(A_n,B_n) \to (A,B)$ in $\Words{\Ccategory}$, we include in $\Sspecies_\bin$:
\begin{itemize}
\item $n$ new colors $I_{x,0},\dots,I_{x,n-1}$, with $\phi_\bin(I_{x,i-1}) = (A,A_i)$ for $1 \le i \le n$;
\item one nullary node $x_0: I_{x,0}$, with $\phi_\bin(x_0) = w_0$;
\item $n$ binary nodes $x_1,\dots,x_n$, where $x_i : I_{x,i-1},R_i \to I_{x,i}$ and $\phi_\bin(x_i) = \id[A]{-}\id[A_i]{-}w_i$ for all $1 \le i \le n$, under the convention that $I_{x,n} = R$.
\end{itemize}
We define the functor $B : \FreeOperad{\Sspecies} \to \FreeOperad{\Sspecies_\bin}$ on colors by $B(R) = R$, on nullary nodes by $B(c) = c$, and on nodes $x : R_1,\dots,R_n \to R$ of positive arity by
$B(x) = x_n\circ_0 \dots \circ_0 x_1 \circ_0 x_0$.
By induction on $n$, there is a one-to-one correspondence between nodes $x : R_1,\dots,R_n \to R$ of $\Sspecies$ and operations $B(x) : R_1,\dots,R_n \to R$ of $\FreeOperad{\Sspecies_\bin}$, so the functor $B$ is fully faithful.

\subsection*{Proof of Prop.~\ref{proposition/pullback-along-ULF}}

Suppose $p_\Qoperad : \Qoperad \to \Ooperad$ is a ULF functor of operads, and consider the pullback of $\phi : \Sspecies \to \ForgetOperad{\Ooperad}$ along $\ForgetOperad{p_\Qoperad} : \ForgetOperad{\Qoperad} \to \ForgetOperad{\Ooperad}$ in the category of species:
\[
  \begin{tikzcd}[column sep=2em,row sep=.4em]
{\Sspecies'}\arrow[dd,"{\phi'}"{swap}]\arrow[rr,"\psi"] && {\Sspecies}\arrow[dd,"\phi"]
\\
& pullback &
\\
{\ForgetOperad{\Qoperad}}\arrow[rr,"{\ForgetOperad{p_{\Qoperad}}}"{swap}] && {\ForgetOperad{\Ooperad}}
\end{tikzcd}
\]
We wish to show that there is a corresponding pullback in the category of operads:
\[
\begin{tikzcd}[column sep=1.5em,row sep=.4em]
\FreeOperad{\Sspecies'}\arrow[dd,"{p'}"{swap}]\arrow[rr,"\FreeOperad{\psi}"] && \FreeOperad{\Sspecies}\arrow[dd,"p"]
\\
& pullback &
\\
\Qoperad\arrow[rr,"{p_{\Qoperad}}"{swap}] && \Ooperad
\end{tikzcd}
\]
Note that $\FreeOperad{\Sspecies'}$ and $\Sspecies'$ have the same colors, namely pairs $(R,R')$ of a color $R$ in $\Sspecies$ and a color $R'$ in $\Qoperad$ such that $\phi(R) = p_\Qoperad(R')$.
It suffices to show that any pair $(\alpha,\alpha')$ of an operation $\alpha$ of $\FreeOperad{\Sspecies}$ and an operation $\alpha'$ of $\Qoperad$ such that $p(\alpha) = p_\Qoperad(\alpha')$ corresponds to a unique operation $\beta$ of $\FreeOperad{\Sspecies'}$ such that $(\FreeOperad{\psi})(\beta) = \alpha$ and $p'(\beta) = \alpha'$.
Now, by the inductive characterization of $\FreeOperad{\Sspecies}$ (cf.~\S\ref{section/magic-formula}), there are two cases to consider:
\begin{itemize}
\item $\alpha = \id[R]$ is an identity operation.
  Since $p_\Qoperad(\alpha') = p(\alpha) = \id[p(R)]$ and ULF functors have unique liftings of identities, $\alpha'$ must also be an identity $\alpha' = \id[R']$ for some $R'$ such that $p_\Qoperad(R') = p(R)$.
  We take $\beta = \id[(R,R')]$.
\item $\alpha = x\bullet(\gamma_1,\dots,\gamma_n)$ is a formal composition of some $n$-ary node $x$ of $\Sspecies$ with operations $\gamma_1,\dots,\gamma_n$ of $\FreeOperad{\Sspecies}$.
  Since $p_\Qoperad(\alpha') = p(\alpha) = \phi(x) \circ (p(\gamma_1),\dots,p(\gamma_n))$, by ULF there exist unique $\beta',\gamma'_1,\dots,\gamma'_n$ such that $\alpha' = \beta' \circ (\gamma'_1,\dots,\gamma'_n)$ and $p_\Qoperad(\beta') = \phi(x)$ and $p_\Qoperad(\gamma'_1) = p(\gamma_1),\dots,p_\Qoperad(\gamma'_n) = p(\gamma_n)$.
  We take $\beta = (x,\beta') \bullet ((\gamma_1,\gamma'_1),\dots,(\gamma_n,\gamma'_n))$.
\end{itemize}

\subsection*{Proof of Prop.~\ref{proposition/pullback-sliced-contour}.}
Every map of species $\psi:\Sspecies\to\Sspecies'$ induces a naturality square
$$
\begin{tikzcd}[column sep=3em,row sep=1em]
{\FreeOperad{\Sspecies}}
\arrow[dd,"{p_{\Sspecies}}"{swap}]\arrow[rr,"{\FreeOperad{\psi}}"] 
&& 
{\FreeOperad{\Sspecies'}}
\arrow[dd,"{p_{\Sspecies'}}"]
\\
&  &
\\
{\Words{\Contour{{\Sspecies}}}}
\arrow[rr,"{\Words{\Contour{{\psi}}}}"{swap}] &&
{\Words{\Contour{{\Sspecies'}}}}
\end{tikzcd}
$$
in the category of operads where the functors of operads $p_{\Sspecies}$ and $p_{\Sspecies'}$ associated to the universal grammars
are the units of the contour / splicing adjunction, see \S\ref{section/contour-category} and \S\ref{section/unigram}.
By Prop.~\ref{proposition/pullback-along-ULF}, we know that the pullback of $p_{\Sspecies'}$ along $\Words{\Contour{{\psi}}}$
is obtained from a corresponding pullback in the category of species
$$
\begin{tikzcd}[column sep=1em,row sep=.4em]
{\Rspecies}
\arrow[dd,"{\rho}"{swap}]\arrow[rr,"{\psi}"] 
&&
{\Sspecies'}
\arrow[dd,"{\phi_{\Sspecies'}}"]
\\
& pullback &
\\
{\Words{\Contour{{\Sspecies}}}}
\arrow[rr,"{\Words{\Contour{{\psi}}}}"{swap}] &&
{\Words{\Contour{{\Sspecies'}}}}
\end{tikzcd}
$$
The pullback~$\Rspecies$ of the map of species ${\phi_{\Sspecies'}}$ along the ULF functor of operads $\Words{\Contour{{\psi}}}$ is the species with colors defined as triples $(R,(R_1^{\tagU},R_2^{\tagD}))$
where $R$ is a color of $\Sspecies'$ and $R_1$, $R_2$ are colors of $\Sspecies$ such that ${\psi(R_1)=\psi(R_2)=R}$~; and with $n$-ary nodes defined as pairs $(x,f)$
where $x$ is a $n$-ary node in $\Sspecies'$ and $f$ is an $n$-ary operation in $\Words{\Contour{\Sspecies}}$ necessarily of the form $f = (y,0){-}\dots{-}(y,n)$, 
for $y$ the unique $n$-ary node of $\Sspecies$ such that $\psi(y)=x$,
since the map of species $\psi:\Sspecies\to\Sspecies'$ is injective of nodes.
The canonical map of species ${\Sspecies}\to{\Rspecies}$ transports every color $R$ of~$\Sspecies$ to the color~$(R,\psi(R)^{\tagU},\psi(R)^{\tagD})$
and every $n$-ary node $y$ of~$\Sspecies$ to the $n$-ary node $(\psi(y),(y,0){-}\dots{-}(y,n))$ of $\Rspecies$.
From this follows that the canonical map of species ${\Sspecies}\to{\Rspecies}$ is injective on colors and bijective on nodes.
Moreover, there are no nodes in $\Rspecies$ whose colors are outside of the image of $\Sspecies$.
We conclude that the canonical functor of operads $\FreeOperad{\Sspecies}\to\FreeOperad{\Rspecies}$ is fully faithful.

%
%

\end{document}